\documentclass[twoside,11pt]{article}

\usepackage{graphicx, subfigure}
\usepackage[top=1in,bottom=1in,left=1in,right=1in]{geometry}
\usepackage[sort&compress]{natbib} 
\usepackage{amsfonts, amsmath, amssymb, amsthm, constants, bbm}
\usepackage{MnSymbol}
\usepackage{authblk}
\usepackage{enumitem}

\usepackage[ruled]{algorithm2e}

\usepackage{color}
\definecolor{darkred}{RGB}{100,0,0}
\definecolor{darkgreen}{RGB}{0,100,0}
\definecolor{darkblue}{RGB}{0,0,150}
\definecolor{orange}{RGB}{200,100,0}

\usepackage{hyperref}
\hypersetup{colorlinks=true, linkcolor=darkred, citecolor=darkgreen, urlcolor=darkblue}
\usepackage{url}

\newcommand{\R}{\mathbb{R}}

\newcommand{\E}{\mathbb{E}}
\newcommand{\N}{\mathbb{N}}

\newcommand{\1}{\mathbf{1}}
 
\newcommand{\Var}{\operatorname{Var}}

\newcommand{\by}{\mathbf{y}}
\newcommand{\bY}{\mathbf{Y}}
\newcommand{\bA}{\mathbf{A}}
\newcommand{\bS}{\mathbf{S}}

\newcommand{\btheta}{\mathbf{\theta}}
\newcommand{\Unif}{\operatorname{Unif}}
\newcommand{\Bin}{\operatorname{Bin}}
\newcommand{\Ber}{\operatorname{Ber}}
\newcommand{\Geom}{\operatorname{Geom}}

\newcommand{\KL}{\operatorname{KL}}

\renewcommand{\tilde}{\widetilde}
\renewcommand{\d}{\ensuremath {\,\mathrm{d}}}
\renewcommand{\vec}[1]{{\boldsymbol{#1}}}
\renewcommand{\P}{\mathbb{P}}

\numberwithin{equation}{section}
\newtheorem{theorem}{Theorem}[section]
\newtheorem{lemma}{Lemma}[section]

\newtheorem{remark}{Remark}[section]
\newtheorem{conjecture}{Conjecture}[section]

\theoremstyle{remark}

\fussy
\begin{document}
\thispagestyle{empty}

\title{Are there needles in a moving haystack?  Adaptive sensing for detection of dynamically evolving signals}
\author[1]{Rui M. Castro}
\author[2]{Ervin T\'anczos}
\affil[1]{Technische Universiteit Eindhoven} 
\affil[2]{University of Wisconsin - Madison}
\date{}
\maketitle

\begin{abstract}
In this paper we investigate the problem of detecting dynamically evolving signals.  We model the signal as an $n$ dimensional vector that is either zero or has $s$ non-zero components.  At each time step $t\in \mathbb{N}$ the non-zero components change their location independently with probability $p$.  The statistical problem is to decide whether the signal is a zero vector or in fact it has non-zero components.  This decision is based on $m$ noisy observations of individual signal components collected at times $t=1,\ldots,m$.  We consider two different sensing paradigms, namely adaptive and non-adaptive sensing.  For non-adaptive sensing the choice of components to measure has to be decided before the data collection process started, while for adaptive sensing one can adjust the sensing process based on observations collected earlier.  We characterize the difficulty of this detection problem in both sensing paradigms in terms of the aforementioned parameters, with special interest to the speed of change of the active components.  In addition we provide an adaptive sensing algorithm for this problem and contrast its performance to that of non-adaptive detection algorithms.
\end{abstract}

\section{Introduction} \label{sec:intro}

Detection of sparse signals is a problem that has been studied with great attention in the past.  The usual setting of this problem involves a (potentially) very large number of items, of which a (typically) much smaller number \emph{may} be exhibiting anomalous behavior.  A natural question one can ask if it is possible to reliably detect if there are indeed some items showing anomalous behavior?  Questions like this are encountered in a number of research fields.  Some examples include epidemiology where one wishes to quickly detect an outbreak or the environmental risk factors of a disease \citep{Fast_Neill_2003,PermScan_Kulldorff_2005, SpatialScan_Huang_2007, Scan_Kulldorff_2009}, identifying changes between multiple images \citep{Flenner_2011}, and microarray data studies \citep{Pawitan_2005} to name a few.

A common point in the examples above is that even though it is not known which items are anomalous, their identity remains fixed throughout the sampling/measurement process.  However, in certain situations the identity of these items may change over time.  

Consider for instance a signal intelligence setting where one wishes to detect covert communications.  Suppose that our task is to survey a signal spectrum, a small fraction of which may be used for communication, meaning that some frequencies would exhibit increased power.  On one hand we do not know beforehand which frequencies are used, but also the other parties may change the frequencies they communicate through over time.  This means we will be chasing a moving target.  This introduces a further hindrance in our ability to detect whether someone is using the surveyed signal spectrum for covert communications.  

Other motivating examples for such a problem include spectrum scanning in a cognitive radio system \citep{Li_2009,Caromi_2013}, detection of hot spots of a rapidly spreading disease \citep{Shah_2011,Zhu_2013,Luo_2013,Wang_2014}, detection of momentary astronomical events \citep{Thompson_2014} or intrusions into computer systems \citep{Gwadera_2005,Phoha_2007}.  The main question that we aim to answer in this paper is how the dynamical aspects of the signal affect the difficulty of the detection problem.

In the more classical framework of the signal detection problem, inference is based on observations that are collected non-adaptively.  However, dealing with time-dependent signals naturally leads to a setting where measurements can be obtained in a sequential and adaptive manner, using information gleaned in the past to guide subsequent sensing actions.  Furthermore, in certain situations it is impossible to monitor the entire system at once, but instead one can only partially observe the system at any given time.  

It is known that, in certain situations, adaptive sensing procedures can very significantly outperform non-adaptive ones in signal detection tasks \citep{AS_Rui_2012}.  Hence our goal is to understand the differences between adaptive and non-adaptive sensing procedures when used for detecting dynamically evolving signals, in situations where the system can only be partially monitored.

\paragraph{Contributions:}
In this paper we introduce a simple framework for studying the detection problem of time-evolving signals.  Our signal of interest is an $n$-dimensional vector $x_t \in\R^n$, where $t\in\N$ denotes the time index.  We take a hypothesis testing point of view.  Under the null the signal is static and equal to the zero vector for all $t$, while under the alternative the signal is a time-evolving $s$-sparse vector.  At each time step $t\in \N$ we flip a biased coin independently for each non-zero signal component to decide if these will ``move'' to a different location.  Thus, the coin bias $p$ encodes the speed of change of the signal support in some sense.  At each time step we are allowed to select one component of the signal to observe through additive standard normal noise, and we are allowed to collect up to $m$ measurements.  Our goal is to decide whether the signal is zero or not, based on the collected observations.

We present an adaptive sensing algorithm that addresses the above problem, and show it is near-optimal by deriving the fundamental performance limits of any sensing and detection procedure.  We do this in both the adaptive sensing and non-adaptive sensing settings for a range of parameter values $p$ and $s$.  It is easy to see that the above problem can not be solved reliably unless we are allowed to collect on the order of $n/s$ measurements.  When the number of measurements is of this order, we can reliably detect the presence of the signal when the smallest non-zero component scales roughly like $\sqrt{p \log (n/s)}$ in the adaptive sensing setting (Theorems~\ref{thm:upper} and \ref{thm:1sparse_lower}).  In the non-adaptive sensing setting detection is possible only when the smallest non-zero component scales like $\sqrt{\log (n/s)}$ (Theorem~\ref{thm:na_lower}).  Hence, under the adaptive sensing paradigm the speed of change influences the difficulty of the detection problem, with slowly changing signals being easier to detect.  Contrasting this, in the non-adaptive sensing setting the speed of change appears to have no strong effect in the problem difficulty when $m$ is of the order $n/s$.  When the number of measurements $m$ is significantly larger than $n/s$ the picture changes quite a bit, and a theoretical analysis of that case is beyond the contribution of this paper.  Nevertheless we provide some simulation results indicating that, in the non-adaptive sensing setting, the signal dynamics will then influence the detection ability.

Despite its simplicity, the setting introduced in this paper provides a good starting point to understand the problem of detecting dynamically evolving signals.  Although we provide several answers in this setting many questions remain (both technical and conceptual).  We hope that this work opens the door for many interesting and exciting extensions and developments, some of which are highlighted in Section~\ref{sec:end}.  

\paragraph{Related work:} The setting where the identity of the anomalous items is fixed over time has been widely studied in the literature.  Classically this problem has been addressed when each entry of the vector is observed exactly once.  In this context both the fundamental limits of the detection problem and the optimal tests are well understood (see \cite{Minimax_Ingster_2000, Det_Ingster_2002,Baraud_2002,HC_Donoho_2004} and references therein).

The same problem has been investigated in the adaptive sensing setting as well.  In \cite{DS_Haupt_2011} the authors provide an efficient adaptive sensing algorithm for identifying a few anomalous items among a large number of items.  These results were generalized in \cite{malloy:ST} to cope with a wide variety of distributions.  The algorithms outlined in these works can in principle also be used to solve the detection problem, that is where only the presence or absence of anomalous items needs to be decided.  In \cite{Seq_Testing_Limits_Malloy_2011} and \cite{AS_Rui_2012} bounds on the fundamental difficulty of the estimation problem were derived, whereas in \cite{AS_Rui_2012} bounds for the detection problems were provided as well.  

Our work here has a similar flavor to all the above, but tackling the problem when the anomalous items may change positions while the measurement process is taking place.  This brings a new temporal dimension to the signal detection problems referenced above.  Statistical inference problems pertaining time-dependent signals have been investigated in various settings in the past.  However, the papers referenced below only have varying degrees of connection to the problem we are considering, as despite our best efforts, we were only able to find a few instances that resemble our setting.

A setting that has some degree of temporal dependence is the monitoring of multi-channel systems.  This problem was introduced in \cite{Zigangirov_1966} and later revisited in \cite{Klimko_1975} and \cite{Dragalin_1996}.  In this setting each channel of a multi-channel system contains a Wiener process, a few of which are anomalous and have a deterministic drift.  The observer is allowed to monitor one channel at a time with the goal to localize the anomalous channels as quickly as possible.
Although there is a clear temporal aspect to these problems, the anomalous channels identity is unchanged during the process.

Another prototypical example of inference concerning temporal data is change-point detection in a system involving multiple processes.  In this problem we have multiple sensors observing stochastic processes.  After some unknown time a change occurs in the statistical behavior of some of the processes, and our goal is to detect when such a change occurs as quickly as possible.  This setting has been studied in \cite{Hadjiliadis_2008}, a Bayesian version of the problem was investigated in \cite{Raghavan_2010}, while the authors of \cite{Bayraktar_2015} deal with a version of the above problem where only one of the sensors is compromised.

This setting shares similarities to ours, but there are some key differences.  In the change-point detection setting, once a process becomes anomalous it remains so indefinitely.  Since some processes are bound to exhibit anomalous behavior, the goal is to minimize the detection delay.  Contrasting this, in the setting we consider an anomalous process can revert back to the nominal state, and there is a possibility that none of the processes are anomalous at any time.  Hence our goal is to decide between the presence or absence of any anomalous processes over the measurement horizon.

A set of more closely related work is concerned with the spectrum scanning of multichannel cognitive radio systems.  Here the aim is to quickly and accurately determine the availability of each spectrum band of a multi-band system where the occupancy status changes over time.  Alternatively one might only aim to quickly find a single band that is available.  This problem has been studied in \cite{Li_2009} and \cite{Caromi_2013}, in which the authors provide efficient algorithms for the problem at hand.  A very similar problem was investigated in \cite{Zhao_2010_onoff}, where one observes multiple ON/OFF processes and wishes to catch one in the ON state.

Although the underlying models of these problems come very close to the one we consider, these works are also change-point detection problems in spirit.  Hence a similar comment applies here as well, namely that the goal of the algorithms of \cite{Li_2009, Caromi_2013} and \cite{Zhao_2010_onoff} is to detect a change-point while minimizing some notion of regret (such as detection delay or sampling cost), which is somewhat different to the problem we are aiming to tackle.


\paragraph{Organization:} Section~\ref{sec:setup} introduces the problem setup, including the signal and observation models and the inference goals.  In Section~\ref{sec:upper} we introduce an adaptive sensing algorithm and analyze its performance.  Section~\ref{sec:lower} is dedicated to the characterization of the difficulty of the detection of dynamically evolving signals.  In particular we show that the algorithm presented in Section~\ref{sec:upper} is near-optimal, and examine the difference between adaptive and non-adaptive sensing procedures.  In Section~\ref{sec:sim} we present numerical evidence supporting a conjecture on the non-adaptive sensing performance limit in the regime when $m$ is of the order $n/s$.  Concluding remarks and avenues for future research are provided in Section~\ref{sec:end}.



\section{Problem setup} \label{sec:setup}

For notational convenience let $[k]=\{ 1,\dots ,k \}$ where $k\in \N$.  In our setting the underlying (unobserved) signal at time $t$ is a $n$-dimensional vector, where time $t\in\N$ is discrete.  Let $\mu>0$ and denote the unknown signal at time $t\in\N$ by $\vec{x}^{(t)}\equiv\left(x_1^{(t)},\ldots,x_n^{(t)}\right)\in\R^n$, where
\[
x^{(t)}_i = \left\{ \begin{array}{ll}
\mu & \text{ if }i \in S^{(t)}\\
0 & \text{ if } i \notin S^{(t)}
\end{array} \right.  \ ,
\]
and $S^{(t)} \subset [n]$ is the support of the signal at time $t$.  We refer to the components of $\vec{x}^{(t)}$ corresponding to the support $S^{(t)}$ as the \emph{active components} of the signal at time $t$.  In Section~\ref{subsec:signal_model} we model the signal as a random process with the property that, at any time, the number of active components is much smaller than $n$.  

In this idealized model the active components of $\vec{x}^{(t)}$ have all same value, which might seem restrictive at first.  However, when the active components have different signs and magnitudes, the arguments of all the proofs hold throughout the paper with $\mu$ playing the role of the minimum absolute value of the active components.  Although a more refined analysis is likely possible, where the minimum is replaced by a suitable function of the magnitudes of active components, we choose to sacrifice generality for the sake of clarity (see also Remark~\ref{rem:minimum_vs_average} below).

The signal is only observable through $m$ noisy coordinate-wise measurements of the form
\begin{equation} \label{eqn:measurement_model1}
Y_t = x^{(t)}_{A_t} + W_t\ ,\ t\in [m]\ ,
\end{equation}
where $A_t \in [n]$ is the index of the entry of the signal measured at time $t$ and $W_t$ are independent and identically distributed (i.i.d.) standard normal random variables.  In the general adaptive sensing setting $A_t$ is a (possibly random) measurable function of $\{ Y_j ,A_j \}_{j\in [t-1]}$ and $W_t$ is independent of $\{\vec{x}^{(j)},A_j \}_{j\in [t]}$ and $\{Y_j\}_{j\in[t-1]}$.  This means the choice of signal component to be measured can depend on the past observations.  A more restrictive setting is that of non-adaptive sensing, where the choice of components to be measured has to be made before any data is collected.  Formally $A_t$ is independent from $\{Y_j ,A_j \}_{j\in [t-1]}$ for all $t\in[m]$.  

\begin{remark}
This measurement model is very similar to that of \cite{DS_Haupt_2011}, \cite{AS_Rui_2012} and \cite{AS_structured_Rui_ET_2013}, where measurements are of the form
\[
Y_t = x_{A_t} + \Gamma_t^{-1} W_t\ ,\ t=1,2,\dots \ ,
\]
when $x$ is a (time-independent) signal, $A_t$ are as above, and $\Gamma_t \in \R$ represent the precision of the measurements (that can be also chosen adaptively).  

In those papers the authors impose a restriction on the total precision used (and not on the number of measurements).  However, since often the precision is related to the amount of time we have for an observation it is somewhat more appealing to consider fixed precision measurements instead.  See also Remark~\ref{rem:discrete_vs_continuous} for an alternative model closer in spirit to that of the above papers.
\end{remark}

\begin{remark}
Recently \cite{Enikeeva_2015_bump} considered an extension of the classical sparse signal detection problem in which the measurements are heteroscedastic, and derived the asymptotic constants of the detection boundary.  In principle, a model similar in spirit to the one presented in that work could also be considered here as well, by assuming that measurements on active components not only have elevated means, but also variance different to 1.

The ideas of \cite{Enikeeva_2015_bump} can be used to modify our detection procedure (in particular the Sequential Thresholding Test -- see Algorithm~\ref{STT}) to craft a procedure that can deal with measurements of different variances.  However, the question of heteroscedasticity for dynamically evolving signals is too rich to be dealt with in the present work.
\end{remark}


\subsection{Signal dynamics}\label{subsec:signal_model}

We consider what might be the simplest non-trivial stochastic model for the evolution of the signal.  Our goal is to model situations where the signal support $S^{(t)}$ changes ``slowly'' over time.  

For concreteness consider first a particular situation, where we assume that at any time $t$ there is a single active component (so $|S^{(t)}|=1$ for all $t\in\N$).  We model the support evolution as a Markov process: 
the support $S^{(1)}$ is chosen uniformly at random over the set $[n]$ (that is, the active component is equally likely to be any of the $[n]$ components); for $t\geq 1$ we flip a biased coin with heads probability $p\in[0,1]$ independent of all the past, and if the outcome is heads then $S^{(t+1)}$ is chosen uniformly at random over $[n]$, otherwise $S^{(t+1)}=S^{(t)}$.
In words, at each time instant the active component stays in place with probability $1-p$ and ``jumps'' to another location with probability $p$.  Thus when $p=1$ the signal has a new support drawn uniformly at random at each time $t\in\N$, whereas in case $p=0$ the support is chosen randomly at the beginning and stays the same over time.  In general, the parameter $p$ can be interpreted as the speed of change of the support, with larger values corresponding to a faster rate of change.  This basic model of signal dynamics can be easily generalized to multiple active components model as follows.

Let $s\in[n]$ be the sparsity of our signal.  We enforce that $|S^{(t)}|=s$ for $t\in\N$, meaning the signal sparsity does not change over time.  For $t=1$, $S^{(t)}$ is chosen uniformly at random from the set $\left\{S\subseteq [n] : |S|=s\right\}$.  For time $t \geq 1$, we flip $s$ independent biased coins, each corresponding to an active component, to decide which components move and which stay in the same place.  Formally take $p\in[0,1]$ and let $\theta^{(t)}_i \sim \Ber (p)$ be independent for every $i \in [s],\ t\in \N$.  Consider an enumeration of $S^{(t)}$ as $S^{(t)}\equiv \left\{S_i^{(t)}\right\}_{i\in [s]}$.  If $\theta_i^{(t)}=0$ component $S_i^{(t)}$ will also be included in $S^{(t+1)}$, otherwise it will move.  The support set $S^{(t+1)}$ is chosen uniformly at random from the set
\[
\bigg\{ S \subset [n]:\ |S|=s,\ S\cap S^{(t)} = \{ S^{(t)}_i: \theta^{(t)}_i =0 \} \bigg\} \ .
\]
For illustration purposes we provide some simulated results in Figure~\ref{fig:support} ($n$ is chosen quite small for visual clarity only).

\begin{figure}
\centering
\subfigure[]{\includegraphics[scale=0.76]{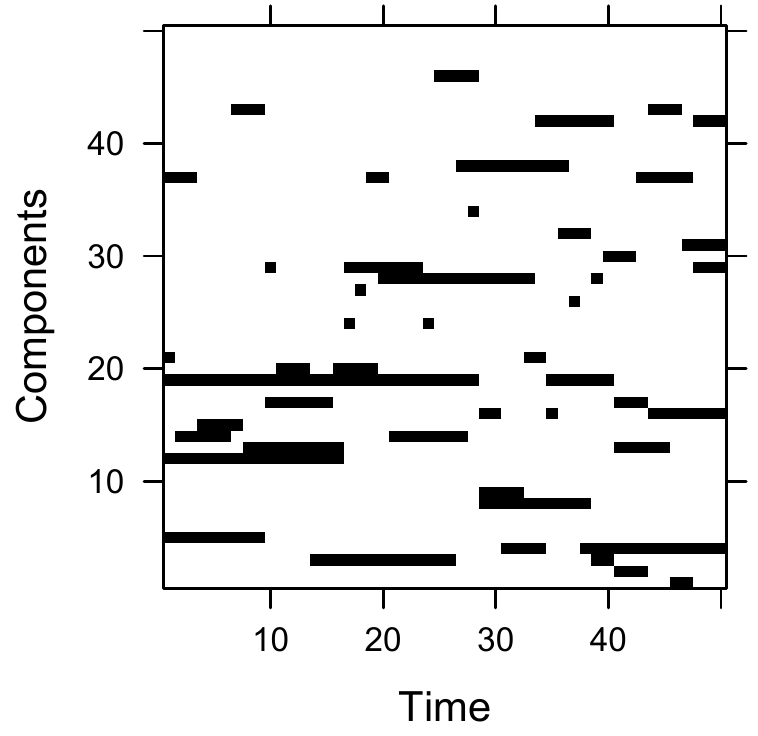}}
\subfigure[]{\includegraphics[scale=0.76]{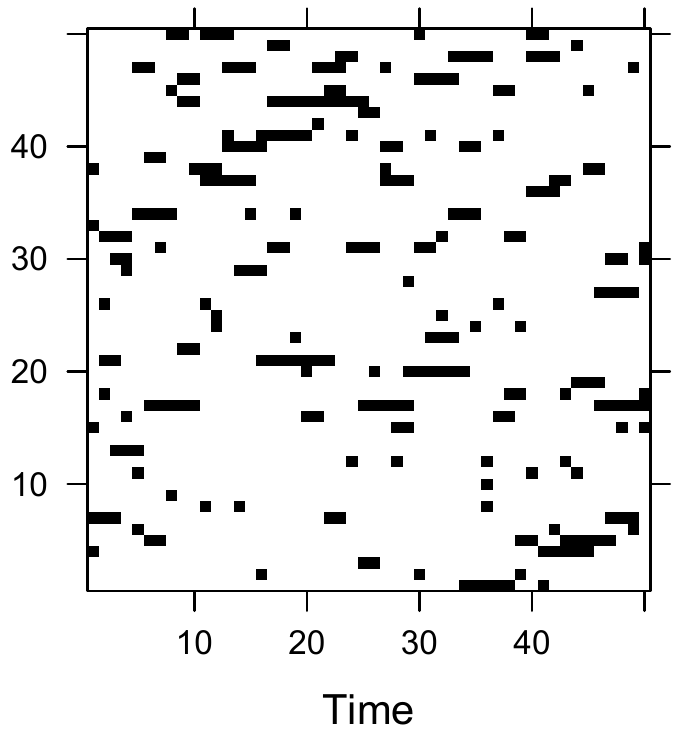}}
\subfigure[]{\includegraphics[scale=0.76]{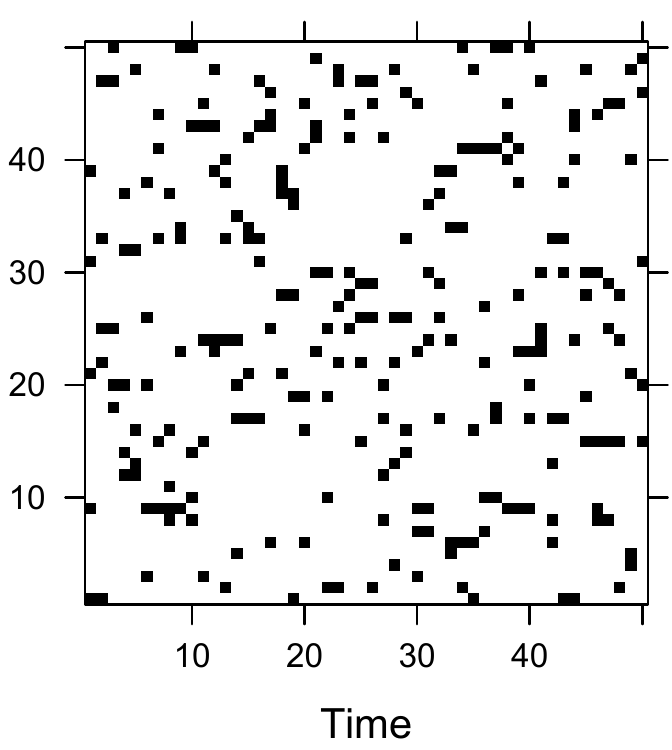}}
\caption{Simulation of the support dynamics with $n=50$, $s=5$ and (a) $p=0.2$; (b) $p=0.5$; and (c) $p=0.8$.  Components in the support are colored black.}
\label{fig:support}
\end{figure}

\begin{remark}\label{rem:discrete_vs_continuous}
Although we consider time to be discrete, continuous-time counterparts of this model are certainly possible (e.g., by taking the transition times to be generated by a Poisson process).  A realistic measurement model in this case would require the variance of the observation noise to be inversely proportional to the time between consecutive measurements, effectively playing a similar role to the precision parameter as in \cite{DS_Haupt_2011, AS_Rui_2012}.
\end{remark}



\subsection{Testing if a signal is present}\label{sec:testing}

In the setting described one can envision several inference goals.  One might try to ``track'' the active components of the signal, attempting to minimize the total number of errors over time.  A somewhat different and in a sense statistically easier goal is to detect the presence of a signal, attempting to answer the question: are there any needles in this moving haystack?  This is the question we pursue in this paper, and it can be naturally formulated as a binary hypothesis test.

Under the null hypothesis there is no signal present, that is $S^{(t)} = \emptyset$ for every $t\in\N$.  Under the alternative hypothesis there is a signal support evolving according to the model described above, for some $s\in[n]$ and $p\in[0,1]$.  Ultimately, after we collected $m$ observations we have to decide whether or not to reject the null hypothesis.  Formally, let $\Psi :\ \{ A_t,Y_t \}_{t\in [m]}\to \{ 0,1 \}$ be a test function where the outcome 1 indicates the null hypothesis should be rejected.  

We evaluate the performance of any test $\Psi\equiv \Psi(\{ A_t,Y_t \}_{t\in [m]})$ in terms of the maximum of the type I and type II error probabilities, which we call the \emph{risk} of a test $R(\Psi)$.  Namely we require
\begin{equation} \label{eqn:test_goal}
R(\Psi)\equiv \max_{i=0,1} \P_i (\Psi \neq i) \leq \varepsilon \ , 
\end{equation}
with some fixed $\varepsilon \in (0,1/2)$, where $\P_0$ and $\P_1$ denote the probability measure of the observations and the null and alternative hypothesis, respectively.  Later on we also use the notation $\E_i$, $i\in\{0,1\}$ to denote the expectation operator under the null and alternative hypothesis respectively.  Note that both the null and alternative hypothesis are simple in the current setup (as we assume $p$ and $\mu$ to be known).  In particular, the density of the observations $\by = (y_1,\dots ,y_m)$ under the alternative can be written as the following mixture:
\[
\d \P_1 (\by ) = \E \left[ \prod_{t\in [m]} g \left( A_t | \{ y_j ,A_j \}_{j\in [t-1]} \right) \left( \1 \{ A_t \in S^{(t)} \} f_\mu (y_t) + \1 \{ A_t \notin S^{(t)} \} f_0 (y_t) \right) \right] \ ,
\]
where $f_\mu$ is the density of a normal distribution with mean $\mu$ and variance 1, $\{ S^{(t)} \}_{t\in [m]}$ are the supports evolving as defined in Section~\ref{sec:setup}, and $g(A_t | \{ y_j ,A_j \}_{j\in [t-1]} )$ is the density of the sensing action at time $t$.  Note, however, that our detection procedures in Section~\ref{sec:upper} do not require knowledge of $\mu$ or $p$.

The main goal of this work is to understand how large the signal strength $\mu$ needs to be, as a function of $n,m,s,p$ and $\varepsilon$ to ensure \eqref{eqn:test_goal} is satisfied.  To this end we first propose a specific adaptive sensing algorithm and evaluate its performance in Section~\ref{sec:upper}.  Furthermore in Section~\ref{sec:lower} we prove that, in several settings, this algorithm is essentially optimal, by showing lower bounds on $\mu$ that are necessary for detection by any sensing and testing strategy.  In the subsequent sections we will see that there is a complex interplay between the parameters $n,m,s$ and $p$ in how they affect the minimum signal strength required for reliable detection.

It is noteworthy to stress that even when we restrict ourselves to the case $p=1$ the nature of the optimal test changes radically depending on the interplay between the remaining parameters $n,m$ and $s$.  In this case, the signal support is reset at every time $t\in\N$, which means that regardless of the sampling strategy (the choice of $A_t$) we are in the situation akin to a so-called sparse mixture model.  These models are now well understood (see \cite{Minimax_Ingster_2000, Det_Ingster_2002}, \cite{Baraud_2002}, \cite{HC_Donoho_2004} and references therein).  We know that in the case of mixture models, for very sparse signals a type of scan test (which is essentially a generalized likelihood-ratio test) performs optimally, whereas for less sparse signals a global test based on the sum of all the observations is optimal.  In our case the interplay between the parameters $n,s$ and $m$ determines the level of sparsity of the sample under the alternative.  This in turn means that when $p=1$ the optimal test and the scaling required for $\mu$, depends on the relation between $m$ and $s/n$.

The above phenomenon becomes even more complex when $p<1$.  Note, however, that unless $m$ is at least of the order of $n/s$ reliable detection is impossible (regardless of the value of $p$).  The reason behind this is that no sampling strategy will sample an active component under the alternative in fewer measurements with sufficiently large probability.  To see this consider the case $p=0$ and suppose there is no observation noise.  Let the sampling strategy be arbitrary and let $\Omega$ denote the event that the algorithm does not sample an active component.  When $m\leq n/s$ we have
\begin{align*}
\P_1 (\Omega ) & \geq \frac{{n-s \choose m}}{{n \choose m}} = \frac{(n-s)(n-s-1)\dots (n-s-m+1)}{n(n-1)\dots (n-m+1)} \\
& \geq \left( 1-\frac{s}{n-m} \right)^m \geq \left( 1-\frac{2s}{n} \right)^{n/s} \ .
\end{align*}
The expression on the right is bounded away from zero when $n/s$ is large enough.  Hence regardless of the sampling strategy, there is a strictly positive probability that no active components are sampled under the alternative, which shows that \eqref{eqn:test_goal} can not hold for $\varepsilon$ smaller than $\left( 1-\frac{2s}{n} \right)^{n/s}$.  When $p> 0$, sampling an active component becomes even harder, hence the same rationale holds.

In this paper we focus primarily on the regime where the number of measurements $m$ is only slightly larger than $n/s$ (what might be deemed to be the ``small sample'' regime).  If we are interested in scenarios where one needs a detection outcome as soon as possible this is the interesting regime to consider.  Interestingly, when $m$ is significantly larger than $n/s$ the optimal sensing and testing strategies, as well as the fundamental difficulty of the problem appears to be quite different than that of the small sample regime, and is an interesting and likely fruitful direction for future work.  In Section~\ref{sec:sim} we conducted a small numerical experiment illustrating how the fundamental performance behavior changes in that regime.

\begin{remark}\label{rem:minimum_vs_average}
The results in this paper can be very naturally generalized for signals with different signs and magnitudes, by considering the class of signals characterized by the minimum signal magnitude.  In the regime where $m$ is of the order of $n/s$ this is essentially the most natural characterization, since only a very small number of active components will actually be observed (so a very low magnitude component will hinder the performance of any method).  When $m$ is significantly larger the picture changes quite significantly and pursuing these results is an interesting avenue for future research beyond the scope of this paper.
\end{remark}


\section{A detection procedure} \label{sec:upper}


In this section we present an adaptive sensing detection algorithm for the setting in Section~\ref{sec:setup} and analyze its performance.  To devise such a procedure we use a similar approach as taken by \cite{AS_structured_Rui_ET_2013} --- first devise a sensible procedure that works when there is no observation noise (i.e., when $W_t\equiv 0$), and then make it robust to noise by using sequential testing ideas.

Consider a setting where there is no measurement noise, that is, when measuring a component of $\vec{x}^{(t)}$ we know for sure whether that component is zero or not.  In such a setting if we find an active component we can immediately stop and deem $\Psi =1$.  Note that it is wasteful to make more than one measurement per component, and that, before hitting an active component, we have absolutely no prior knowledge on the location of active components.  Therefore an optimal adaptive sensing design is random component sampling without replacement.  If we look at a large enough number of randomly chosen components and only observe zeros, it becomes reasonable to conclude that there are no active components and so we deem $\Psi=0$.  Bear in mind though that in case we did not observe any active components we might have simply been unlucky, and missed them even though they are present.  Hence, there is always a possibility for a false negative decision regardless of how many components we observe, unless $p=0$ and $m\geq n-s$.

The procedure that we propose is a ``robustified'' version of the one explained above, so that it can deal with measurement noise.  This is done by performing a simple sequential test to gauge the identity of the component that we are observing.  A natural candidate for this is the Sequential Likelihood-Ratio Test (SLRT), introduced in \cite{SLRT_Wald_1945}.
However, the dynamical nature of the signal causes some difficulties.  In particular the identity/activity of the component that we are observing might change while performing the test, creating many analytic hinderances in the study of the SLRT performance.  We instead use a simplified testing/stopping criteria that is easier to analyze in such a scenario.

The basic detection algorithm, presented in Algorithm~\ref{basic_algorithm}, queries components uniformly at random one after another and tests their identity (whether they are active or not during the subsequent time period) using the sequential test to be described later.  Once a component is deemed to have been active we set $\Psi =1$ and stop collecting data.  If after examining $T$ components or exhausting our measurement budget no components are deemed active we set $\Psi = 0$.  

Formally, let $\{ Q_j \}_{j\in [T]}$ denote the components queried by Algorithm~\ref{basic_algorithm}.  We choose $Q_j,\ j\in [T]$ to be independent $ \Unif ([n])$ random variables.\footnote{In principle one could ensure these are sampled without replacement from $[n]$, but this would only unnecessarily complicate the analysis without yielding significant performance gains.} The appropriate number of queries $T\leq m$ will be chosen later.  For each $Q_j$ we run a sequential test to determine the identity of that component.  We refer to our sequential test as Sequential Thresholding Test (STT).  

To gauge the identity of $Q_j,\ j\in [T]$, the STT algorithm makes multiple measurements at that coordinate.  The exact number of measurements depends on the observed values (in a way we describe in detail later), and hence it is random.  We denote the number of observations collected by STT at coordinate $Q_j$ by $N_j$.  Formally, this means that $A_t = Q_j$ for $t\in \big[ 1+ {\sum_{i=1}^{j-1}} N_i , {\sum_{i=1}^j} N_i \big]$.

At the end of the $j$th run of STT ($j=1,2,\dots ,T$), the STT returns either that an active component was present at coordinate $Q_j$, or that no active component was present at that location.  In the former case there is no need to collect any more samples: Algorithm~\ref{basic_algorithm} stops and declares $\Psi =1$.  Otherwise we continue with applying STT to coordinate $Q_{j+1}$.  If all $T$ runs of STT found no signal, or we exhaust our measurement budget, Algorithm~\ref{basic_algorithm} stops and returns $\Psi =0$.


\begin{algorithm}[h]\label{basic_algorithm}
\caption{Detection Algorithm}
\SetKw{KwParameters}{Parameters:}
\SetKw{KwInitialization}{Initialization:}
\SetKw{KwReturn}{Terminate:}
\KwParameters{}\\
$\quad$\textbullet\ Number of queries $T \in \N$\\
$\quad$\textbullet\ Queries $Q_1,\dots ,Q_T \displaystyle{\mathop{\sim}^{iid}} \Unif ([n])$\\
\For{$j\leftarrow 1$ \KwTo $T$}
	{
	 Perform a STT for the component indexed by $Q_j$\\
	 If the STT returns \textbf{``Signal"}: set $\Psi =1$ and \textbf{break}\\
	 If measurement budget is exhausted: set $\Psi =0$ and \textbf{break}\\
	}	
\end{algorithm}

The sequential test that we use to examine the identity of a queried component is based on the ideas of distilled sensing introduced and analyzed in \cite{DS_Haupt_2011} and the Sequential Thresholding procedure of \cite{malloy:ST}.  The distilled sensing algorithm is designed to recover the support of a sparse signal (whose active components remain the same during the sampling process).  The main idea there is to use the fact that the signal is sparse and try to measure active components as often as possible, while not wasting too many measurements on components that are not part of the support.  Our aim here is somewhat similar: on one hand we wish to quickly identify when the component that we are sampling is non-active so that we can move on to probe a different location of the signal.  On the other hand in case we are sampling an active component we wish to keep sampling it as long as it is active to collect as much evidence as possible.  However, unlike in the original setting of distilled sensing, we need to be able to quickly detect that we are sampling an active component, as it will eventually move away because of the dynamics.  To address the last point the STT algorithm in Algorithm~\ref{STT} uses an evolving threshold for detection depending on the number of observations collected.

We present STT in a way that emphasizes that it is a stand-alone routine plugged into the detection algorithm above, and not necessarily specific to the problem at hand.  Hence, when discussing STT, the observations the STT makes are denoted by $X^{(1)},X^{(2)},\dots\ $.  In the context of Algorithm~\ref{basic_algorithm}, for the $j$th call of STT we have $X^{(1)},X^{(2)},\dots$ to be independent normal random variables with variance one and means respectively $x^{(T_j)}_{Q_j} x^{(T_j +1)}_{Q_j},\dots$, where $T_j=1+\sum_{i=1}^{j-1} N_i$.

\begin{algorithm}[h]\label{STT}
\caption{Sequential Thresholding Test (STT)}
\SetKw{KwParameters}{Parameters:}
\SetKw{KwInitialization}{Initialization:}
\SetKw{KwReturn}{Terminate:}
\KwParameters{}\\
$\quad$\textbullet\ $k \in \N ,\ t_1 > t_2 > \dots > t_k > 0$\\
$\quad$\textbullet\ STT \underline{can} sequentially observe $X^{(1)},X^{(2)},\dots ,X^{(k)}$\\
\For{$j\leftarrow 1$ \KwTo $k$}
	{
	 Observe $X^{(j)}$ and compute $\overline{X}^{(j)} = \sum_{i=1}^j X^{(i)} /j$\\
	 If $\overline{X}^{(j)} \leq t_k$: \textbf{break} and declare \textbf{No signal}\\
	 If $\overline{X}^{(j)} > t_j$: \textbf{break} and declare \textbf{Signal}\\
	}
\end{algorithm}

In words, STT collects at most $k$ measurements sequentially and keeps track of the running average until one of the stopping conditions is met.  The first stopping condition says that once the running average drops below the threshold $t_k$ we stop and declare that there is no signal present.  The second says that if the running average at step $j$ exceeds a threshold $t_j$, we stop and conclude that a signal component is present.  Note that after each measurement the upper threshold decreases, eventually reaching $t_k$, hence the procedure necessarily terminates after at most $k$ measurements.

Key to the performance of the STT is a good choice of $k$ and $\{ t_j \}_{j\in [k]}$, which is informed by the following heuristic argument:  the sample collected by the detection algorithm consists of $T$ blocks of measurements, where each block corresponds to an application of STT.  Let the block lengths be denoted by $\{ N_j \}_{j\in [T]}$.  Suppose for a moment that blocks entirely consist of either zero mean or non-zero mean measurements.  In this case we can simply think of each block $j$ as a single measurement with mean multiplied by $\sqrt{N_j}$ for all $j\in [T]$.  This would reduce the problem to a detection problem in a $T$-dimensional vector, each component being normally distributed and having unit variance.  This is a well-understood setting, and we know that in this case the signal strength needs to scale as $\sqrt{\log T}$ when there are not too many active components (see for instance \cite{HC_Donoho_2004} and the references therein).  Recall that we are concerned with the case where the number of measurements we are allowed to make is of the order $n/s$.  Hence we do not expect to encounter active components too many times.  This heuristic shows that we should calibrate STT in a way that when it encounters $j$ consecutive measurements with elevated mean, it should be able to detect it when $\mu \approx \sqrt{\tfrac{1}{j} \log T}$\footnote{In this informal discussion, the notations $\approx$ and $\gtrsim$ hide constant factors and/or $\log (1/\varepsilon)$ terms.}.  Furthermore, considering the tail properties of the Gaussian distribution, it is easy to see that we also need $\mu \gtrsim \sqrt{\log \tfrac{1}{\varepsilon}}$ for reliable detection.  Recalling that $j\leq k$, this shows that choosing $k$ greater than $\log T$ does not buy us anything.  Informed by the above heuristic argument we choose the parameters of STT so that the following result holds.

\begin{lemma}\label{lem:STT}
Let $\varepsilon \in (0,1)$ and define the parameters of STT as
\begin{align*}
k & = \lfloor \log (T/2) \rfloor \ , \\
t_j & = \sqrt{\frac{c(2\varepsilon /T)}{j} \log \frac{2T}{\varepsilon}} ,\ j\in [k] \ ,
\end{align*}
where
\[
c(x)= 2\left( 1+ \frac{\log \log (1/x)}{\log (1/x)} \right) .
\]
Denote the observations available to the STT by $X^{(1)},\dots ,X^{(k)}$ (note that the STT may terminate without observing all the variables).  Then the following holds:
\begin{enumerate}
\item[(i)]{If $X^{(i)} \displaystyle{\mathop{\sim}^{\text{i.i.d.}}} \mathcal{N}(0,1)$ for $i\in [k]$, then STT declares ``Signal" with probability at most $\varepsilon /T$.}
\item[(ii)]{For any $j\in[k]$, if the $X^{(i)} \displaystyle{\mathop{\sim}^{\text{i.i.d.}}} \mathcal{N}(\mu ,1)$ for $i\in [j]$ with
\[
\mu \geq \sqrt{\frac{c(2\varepsilon /T )}{j} \log \frac{2T}{\varepsilon}} + \sqrt{2 \log \frac{4}{\varepsilon}} \ ,
\]
then STT declares ``No Signal" with probability at most $\varepsilon/3$.}
\end{enumerate}
\end{lemma}

Note that, for (ii) it suffices for the first $j$ observations to have elevated mean to guarantee the good performance of the STT.

\begin{proof}[Proof of Lemma~\ref{lem:STT}]
For the first part suppose note that the STT declares ``Signal" if at any time step $j\in [k]$ the running average $\overline{X}_j$ exceeds the threshold $t_j$.
\begin{align*}
\P \left( \exists j \in [k]: \overline{X}^{(j)} \geq t_j \right) & \leq \sum_{j=1}^{k} \P (\overline{X}^{(j)} \geq t_j) \\
& \leq \sum_{j=1}^{k} \frac{1}{2} \exp \left( - \frac{j t_j^2}{2} \right) \\
& = \sum_{j=1}^{\lfloor \log (T/2) \rfloor} \frac{1}{2} \exp \left( -\frac{c(2\varepsilon /T)}{2} \log \frac{T}{2\varepsilon} \right) \\
& \leq \frac{1}{2}\log (T/2) \cdot \left( \frac{2\varepsilon}{T} \right)^{c(2\varepsilon /T)/2} \ ,
\end{align*}
where the first inequality follows by a union bound, and the second inequality is follows by a tail bound on Gaussian random variables noting that $\overline X_j\sim\mathcal{N}(0,1/j)$.  The last expression above is at most $\varepsilon /T$, which can be checked by taking the logarithm:
\begin{align*}
\log \left( \frac{1}{2} \log (T/2) \cdot \left( \frac{2\varepsilon}{T} \right)^{c(2\varepsilon /T)/2} \right) &= \log \log (T/2) + \left( 1- \frac{\log \log (T/(2\varepsilon))}{\log (2\varepsilon/T)} \right) \log (2\varepsilon /T) -\log 2\\
& = \log \log (T/2) + \log (2\varepsilon/T) - \log \log (T/(2\varepsilon)) -\log 2\\
& \leq \log \frac{\varepsilon}{T} \ .
\end{align*}

For the second part assume the conditions in (ii) hold for $\mu$ as given in the lemma.  Define the event
\[
\Omega = \left\{ \exists i\in [j-1]: \overline{X}^{(i)} \leq t_k \right\} \ .
\]
Note that if this event happens, we stop and declare ``No signal" in one of the first $j-1$ steps.
\begin{align*}
\P (\textrm{Declare ``No signal"}) & = \P (\Omega ) + \P (\textrm{Declare ``No signal"})|\overline \Omega )\P ( \overline{\Omega} ) \\
& \leq \P (\Omega ) + \P ( \overline{X}^{(j)} \leq t_j|\overline\Omega )\P ( \overline{\Omega} ) \\
& \leq \P (\Omega ) + \P(\overline{X}^{(j)} \leq t_j ) \ .
\end{align*}
Using a union bound and the same Gaussian tail bound as before, the last expression can be upper bounded by
\begin{equation}\label{eqn:miss}
\sum_{i=1}^{j-1} \frac{1}{2} \exp \left( -\frac{i (\mu -t_k)^2}{2} \right) + \frac{1}{2} \exp \left( - \frac{j(\mu -t_j)^2}{2} \right) \ .
\end{equation}

Considering the first term above, note that
\[
\mu - t_k \geq t_j + \sqrt{2 \log \frac{4}{\varepsilon}} -t_k \geq \sqrt{2 \log \frac{4}{\varepsilon}} \ ,
\]
since $t_j \geq t_k$ (recall that $j\leq k$).  Hence the first term can be upper bounded as
\[
\sum_{i=1}^{j-1} \frac{1}{2} \exp \left( -\frac{i (\mu -t_k)^2}{2} \right) \leq \frac{1}{2} \sum_{i=1}^{j-1} (\varepsilon /4)^i \leq \frac{\varepsilon}{2} \frac{1}{4-\varepsilon} \leq \varepsilon /6 \ .
\]
On the other hand, when $\mu$ satisfies the inequality above, the second term is simply upper bounded by $(\varepsilon /4)^j$, and so the left-hand-side of \eqref{eqn:miss} is less than $\varepsilon/6+\varepsilon/8<\varepsilon/3$.
\end{proof}

Using Lemma~\ref{lem:STT}, we can establish a performance guarantee for our detection algorithm.  Though it is possible to derive a result for fixed $n$ and $s$ it is more transparent to state a result for large $n$ instead, better highlighting the impact of parameter $p$.  Keeping this comment in mind, note that $2 \leq c(x) \leq 2(1+1/e)\leq 2\sqrt{2}$ and $c(x)\to 2$ as $x \to 0$.  Thus, keeping $\varepsilon$ fixed and letting $T\to \infty$, we see that if there exists a $\tau >1$ for which
\[
\mu \geq \tau \sqrt{\frac{2}{j} \log T} + \sqrt{2 \log \frac{4}{\varepsilon}} \ ,
\]
then for $T$ large enough the condition on $\mu$ in Lemma~\ref{lem:STT} is satisfied.  Furthermore, recall that our main interest is how the algorithm performs when the time horizon (number of measurements) is only slightly larger than $n/s$.





\begin{theorem}\label{thm:upper}
Fix $\varepsilon \in (0,1/3)$ and assume $s\equiv s_n=o(n/(\log n)^2)$ as $n\to \infty$.  The parameter $p\equiv p_n$ is also allowed to depend on $n$.  Set $T= \tfrac{9n}{2s} \log_2 \tfrac{3}{\varepsilon}$ and the parameters of STT according to Lemma~\ref{lem:STT}.  If the measurement budget is $m\geq 2T$ the detection algorithm satisfies
\[
R(\Psi)=\max_{i=0,1} \P_i (\Psi \neq i) \leq \varepsilon \ ,
\]
whenever
\[
\mu \geq \tau \sqrt{2 \max \{ 2p,\tfrac{1}{\log (n/s)} \} \log (n/s)} + \sqrt{2 \log \frac{4}{\varepsilon}} \ ,
\]
for $n$ large enough and $\tau>1$ fixed (but arbitrary).
\end{theorem}

Before we move on to the proof of this result, let us discuss its message.  First note that the detection algorithm is agnostic about the speed of change $p$ and the signal strength $\mu$, though it does require knowledge of the sparsity $s$ to set the parameter $T$.  

The number of measurements that we require is a multiple of $n/s$, which is the minimum amount necessary to be able to solve the problem (see Section~\ref{sec:testing}).  Furthermore, when $p < 1/(2\log (n/s))$ the signal strength needs to scale as $\sqrt{\log (1/\varepsilon )}$, and when $p\geq 2/\log (n/s)$ it needs to scale as $\sqrt{p \log (n/s)}$.  This matches the intuition that the speed of change $p$ affects the problem difficulty in a monotonic fashion.  We will show in Section~\ref{sec:lower} that in the regime $m \approx n/s$ this scaling of $\mu$ is necessary to reliably solve this detection problem.

In Figure~\ref{fig:STT_detection} we present an illustration of the above detection algorithm.  We can clearly see the ``random'' exploration (in red) and the ``tracking'' of active components (in green).  Note that in this case the algorithm missed that an active component was hit at time 8, so more exploration was needed.

\begin{figure}[t]
\centering
\includegraphics[scale=0.9]{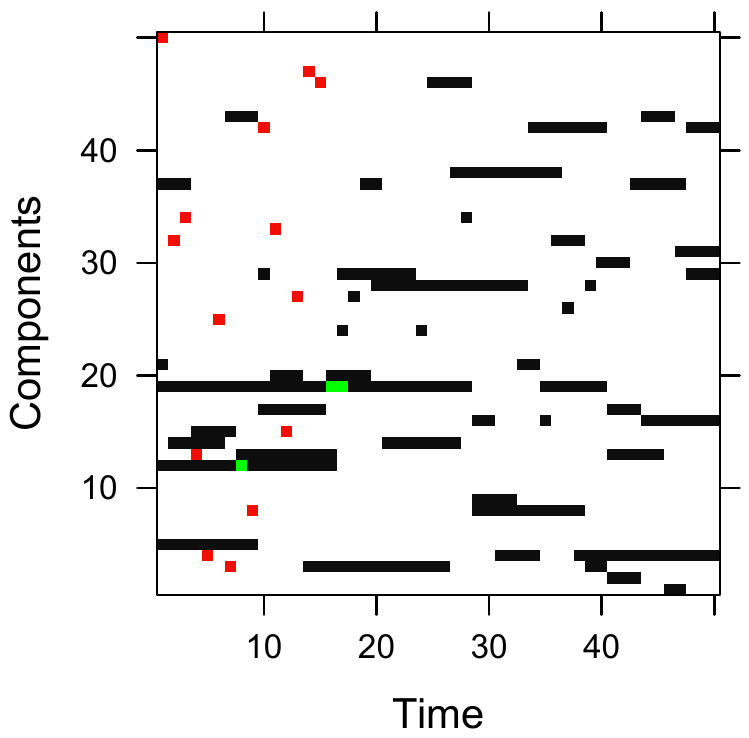}
\caption{Simulation of support dynamics and detection algorithm with $n=50$, $s=5$, and $p=0.2$ (corresponding to the same signal realization as in Figure~\ref{fig:support}).  The algorithm was ran as prescribed by the Theorem~\ref{thm:upper} with $\epsilon=0.05$ and the signal strength $\mu$ is given by the expression in the same theorem with $\tau=1$.  The active signal components are in black.  In red and green are sampled non-active and active signal components, respectively.  The detection algorithm deems that a signal is present after 18 measurements.}
\label{fig:STT_detection}
\end{figure}

\begin{remark}\label{rem:m_regime}
As we have mentioned in Section~\ref{sec:testing}, for now we are interested in the case where the number of observations we can make is roughly $n/s$.  Note that Theorem~\ref{thm:upper} claims the same performance guarantee for every $m$ that is at least of order $n/s$.  

In fact, it is not hard to see that the performance of this algorithm does not improve as $m$ increases, hinting that it is suboptimal for large $m$.  Actually this algorithm completely ignores the fact that a component might have multiple periods of activity over time, and that activity evidence from multiple components might be combined for detection, in a more global fashion.  

Consider the following simple algorithm: sample components uniformly at random in each step $t\in [m]$.  Then in each step we hit an active component with probability $s/n$.  We then roughly have $ms/n$ active components in our sample under the alternative.  Consider the standardized sum of our observations.  Under the null this follows a standard normal distribution, whereas under the alternative it is distributed as $N(\sqrt{m} s \mu/n,1)$.  

Thus reliable detection using this simple global algorithm is possible when $\mu$ is of the order $n/(\sqrt{m} s)$.  Hence this algorithm clearly outperforms the one above when $m$ is large enough (compared to $n/s$).  This phenomena is not unlike that present in sparse mixture detection problems (e.g.  as in \cite{Minimax_Ingster_2000}) where depending on the sparsity a global test might be optimal.
\end{remark}

\begin{proof}[Proof of Theorem~\ref{thm:upper}]
In light of Lemma~\ref{lem:STT}, the type I error probability is at most $\varepsilon$ by a union bound.  Hence we are left with studying the alternative.

There are two ways that our algorithm can make a type II error.  Either the measurement budget is exhausted, or we fail to identify an active component in $T$ runs of STT.  We bound the probability of the first event by $\varepsilon/3$, and of the second event by $2\varepsilon/3$ ensuring that under the alternative the probability of error is bounded by $\varepsilon$.

We start with upper bounding the probability of exhausting our measurement budget.  Let $N_j$ denote the number of measurements that STT makes when called for the $j$th time, for $j\in [T]$.  Note that these random variables are independent and identically distributed, because the components to query are selected uniformly at random independently from the past, the dynamic evolution of the model is memoryless, and the observation noise is independent.  First we upper bound $\E_1 (N_1)$.  Note that $1 \leq N_1 \leq k$, where $k=\lfloor \log (T/2) \rfloor$ by Lemma~\ref{lem:STT}.  Let $\Omega$ denote the event that a non-zero mean observation appears at location $A_1$ in any of the first $k$ steps.  By the law of total expectation we have
\[
\E_1 (N_1) \leq k \P_1 (\Omega ) + \E_1 (N_1 |\overline{\Omega} ) \ .
\]
Note that
\begin{align*}
\P_1 (\Omega ) & = \P_1 (\exists t\in[k]: A_1 \in S^{(t)} ) \leq \sum_{t=1}^k \P_1 (A_1 \in S^{(t)} ) \\
& \leq \frac{s}{n} + (k-1)\frac{s}{n-s} \leq \frac{ks}{n-s} \ ,
\end{align*}
since the choice of $A_1$ (and $S^{(1)}$) is random, and in each subsequent step the probability that a signal component moves to location $A_1$ is at most $s/(n-s)$ regardless of $p$.  On the other hand, recalling that $t_k = \sqrt{\tfrac{c(2\varepsilon /T)}{k} \log \tfrac{T}{2\varepsilon}} \geq \sqrt{2}$ is the lower stopping boundary of STT,
\begin{align*}
\E_1 (N_1|\overline{\Omega}) & = 1+\sum_{t=2}^k \P_0 (N_1 \geq t)\\
& \leq 1+\sum_{t=2}^k \P_0 (\overline{X}_{t-1} > t_k) \leq 1+\sum_{t=2}^k \P_0 (\overline{X}_{t-1} > \sqrt{2}) \\
& \leq 1+ \frac{1}{2} \sum_{t=1}^{k-1} e^{-t} \leq 1+ \frac{1}{2(e-1)} < 3/2 \ .
\end{align*}
Hence
\[
\E_1 (N_1) \leq 1+ \frac{1}{2(e-1)} + \frac{k^2s}{n-s} < 3/2 \ ,
\]
for large enough $n$, since the last term can be made arbitrarily small by the definition of $T$, and the assumption on $s$.  Since $N_1$ is also a bounded random variable, an easy (but crude) way of proceeding is to use Hoeffding's inequality to get
\begin{align*}
\P_1 \left( \sum_{j=1}^T N_j > m \right) & = \P_1 \left( \sum_{j=1}^T N_j - \E_1 \Big( \sum_{j=1}^T N_j \Big) > m - \E_1 \Big( \sum_{j=1}^T N_j \Big) \right) \\
& \leq \P_1 \left( \sum_{i=1}^T N_i - \E_1 \Big( \sum_{i=1}^T N_i \Big) > T/2 \right) \\
& \leq \exp \left( - \frac{T}{2k^2} \right)=\exp \left( - \frac{T}{2\lfloor\log(T/2)\rfloor^2} \right)\leq \varepsilon/3 \ ,
\end{align*}
provided $T$ is large enough, which is the case if $n$ is large enough.  This shows that the probability that the measurement budget is exhausted is bounded by $\varepsilon/3$.

The final step in the proof is to guarantee that the algorithm identifies an active component in one of the $T$ tests with high probability.  To show this, we first guarantee that there will be an instance in the repeated application of STT where the first $1/(2p)$ observations that the procedure has access to have elevated mean (when $p=0$ we only need that the STT probes an active component at least once).  Then we can apply Lemma~\ref{lem:STT} together with a union bound to conclude the proof.

Let $T_j = 1+\sum_{i=1}^{j-1} N_i$ denote the time when STT starts for the $j$th time.  Let $N = \sum_{j=1}^T \1 \{ Q_j \in S^{(T_j)} \}$ denote the number of times an active component is sampled at the start of an STT.  Note that $N \sim \Bin (T,s/n)$.  In these situations the STT has access to a sequence of active measurements (of random length).  Denote the number of consecutive active observations these STTs have access to by $\{ \eta_i \}_{i\in [N]}$, and for now assume $p>0$.  Note that $\eta_i \sim \Geom (p)$ and $\{ \eta_i \}_{i\in [N]}$ are independent.  We have
\begin{align*}
\P (\forall i \in [N]: \eta_i < 1/(2p) ) & \leq \P \left( \forall i\in [N]: \eta_i <1/(2p)\ | N\geq \log_2 \tfrac{3}{\varepsilon} \right) \\
& \qquad + \P \left( N < \log_2 \tfrac{3}{\varepsilon} \right) \ .
\end{align*}
On one hand, note that the median of $\eta_i$ is $\lceil -1/\log_2 (1-p) \rceil$ which is greater than $1/(2p)$.  This can be easily checked by considering the cases $p\geq 1/2$ and $p<1/2$ separately.  Hence the first term above can be upper bounded as
\[
\P \left( \forall i\in [N]: \eta_i <\lceil -1/\log_2 (1-p) \rceil\ | N\geq \log \tfrac{3}{\varepsilon} \right) \leq 2^{-\log_2 \tfrac{3}{\varepsilon}} = \varepsilon /3 \ .
\]

On the other hand, $N \sim \Bin (T,s/n)$ and so by Bernstein's inequality,
\[
\P \left( N < (1-\delta ) \frac{Ts}{n} \right) \leq \exp \left( -\frac{3 \delta^2}{8} \frac{Ts}{n} \right) \ ,
\]
for any $\delta \in (0,1)$.  However, note that plugging in the value of $T$ together with $\delta = 2/3$ yields
\[
\P \left( N<\log_2 \tfrac{3}{\varepsilon} \right) = \P\left( N < (1-\delta)\frac{Ts}{n}\right)\leq \exp \left( -\frac{49}{48} \log_2 \tfrac{3}{\varepsilon} \right) < \varepsilon /3 \ ,
\]
since $\log_2 x > \log x$ for $x>1$.  So we conclude that the probability that there is no block (out of $T$) with the first $1/(2p)$ observations active is bounded by $2\varepsilon/3$.  When $p=0$, we only need to control $\P (N=0)$, for which we can simply use the inequality above since $\log_2 \tfrac{3}{\varepsilon} >0$.  

Finally, if such a block is present the probability STT will not detect it is bounded by $\varepsilon/3$ via part (ii) of Lemma~\ref{lem:STT}, provided
\[
\mu\geq \sqrt{\frac{c(2\varepsilon/T)}{\min \{ 1/(2p),\lfloor \log (T/2) \rfloor \} }\log\left(\frac{T}{2\varepsilon}\right)}+\sqrt{2\log\frac{4}{\varepsilon}}\ ,
\]
where one should note that the blocks sampled by the STT are never larger than $\lfloor \log (T/2) \rfloor$.  It is easily checked that the above condition is met for the choices in the theorem, provided $n$ is large enough, concluding the proof.
\end{proof}


\section{Lower bounds} \label{sec:lower}

In this section we identify conditions for the signal strength that are necessary for the existence of a sensing procedure to have small risk, namely
\begin{equation}\label{eqn:error}
R(\Psi)=\max_{i=0,1} \P_i (\Psi \neq i) \leq \varepsilon \ .
\end{equation}
We consider first the non-adaptive sensing setting.  This is done both for comparison purposes (to highlight the gains of sensing adaptivity) but also illustrates some of the interesting features of this problem.  In this case the sensing procedure is simply the choice of when and where to measure a component, before any data is collected.  Then we consider the adaptive sensing setting to show the near-optimality of the algorithm proposed in Section~\ref{sec:upper}.  In both cases our primary interests in on the regime $m \approx n/s$, as highlighted in Section~\ref{sec:testing}.


\subsection{Non-adaptive sensing} \label{sec:lower_na}

In the non-adaptive sensing setting, the sampling strategy $\{ A_t \}_{t\in [m]}$ needs to be specified before any observations are made.  Note that this does not exclude the possibility of having a random design of the sensing actions.

Common sense tells us that supports that are changing fast are harder to detect than those that are changing slowly, provided all other parameters are fixed.  In other words, the problem difficulty should be increasing in the parameter $p$, meaning the signal magnitude $\mu$ needed to ensure \eqref{eqn:error} should grow monotonically in $p$.  Formalizing this heuristic in general turns out to be technically challenging with the methodologies we are aware of.  Because of this we focus on the two extreme cases: when the signal is static ($p=0$), and when the entire signal resets at each time instance ($p=1$).

\begin{remark}
Note that in the case $s=1$ it is relatively easy to formalize that the problem difficulty is non-decreasing in $p$.

Suppose there exists an algorithm (denoted by \textbf{Alg}) that performs accurate detection for some $p>0$, and suppose we need to perform the detection task of a static signal.  The idea is to transform the signal into one that has the same distribution as if it were generated according to the model of Section~\ref{subsec:signal_model} with parameter $p$, and apply \textbf{Alg} to the modified signal.  If such a transformation is possible than the existence of \textbf{Alg} implies the existence of an accurate detection procedure -- in other words, the problem difficulty is non-decreasing in $p$.

Such a transformation is easy to construct for $s=1$, in fact one can almost follow the description of the signal model of Section~\ref{subsec:signal_model} word-by-word.  Let $\{ \theta_t \}_{t\in [m-1]}$ be i.i.d.  $\Ber (p)$ variables and w.l.o.g.  $\theta_m =1$ --- these represent the coin flips in the description of Section~\ref{subsec:signal_model}.  Let $N=\sum_{t\in [m]} \1 \{ \theta_t =1 \}$ be the number of times the coin came up heads and $\tau_0 =0$ and $\tau_j = \inf \{ t> \tau_{j-1}:\ \theta_t =1 \},\ j\in [N]$ be the instances when the coin came up heads.  Finally, let $\{ \pi_i \}_{i\in [N]}$ be permutations of $[n]$ drawn independently and uniformly at random (from the set of possible permutations).

It is clear that a static support that is permuted by $\pi_i$ on the time intervals $[\tau_{i-1}+1,\tau_i ]$ will "look" like a support sequence evolving with parameter $p$.  Formally, one can show that if $\bS \equiv \{ S^{(t)} \}_{t\in [m]}$ is a static support sequence (chosen uniformly at random) then $\tilde{\bS} \equiv \{ \tilde{S}^{(t)} \}_{t\in [m]}$ defined as
\[
\tilde{S}^{(t)} = \sum_{i\in [N]} \1 \{ t\in [\tau_{i-1}+1 ,\tau_i ] \pi_i (S^{(t)})
\]
is distributed as a support sequence generated according to the model described in Section~\ref{subsec:signal_model} with parameter $p$.  Hence for $s=1$ the problem difficulty is indeed non-decreasing in $p$.

Nonetheless the authors did not find an obvious way to extend this argument to general sparsities, because the signal components change their locations at possibly different times.  We note at this point that if one considered a more restrictive model where the entire support of the signal would reset simultaneously (a setting perhaps not vastly different to the one we are considering) would enable an argument similar to the above.
\end{remark}

We have the following result for these two extreme cases, which we prove at the end of the section.  Note that these are not asymptotic, and hold for any $n,m$ and $s$ satisfying the assumptions in the statement.

\begin{theorem}\label{thm:na_lower} Let $n,s,m \in \N$ be fixed (with $s\leq n$), consider a setup described in Section~\ref{sec:setup}, and suppose there is a non-adaptive sensing design and a test $\Psi$ satisfying
\[
R(\Psi)=\max_{i=0,1} \P_i (\Psi \neq i) \leq \varepsilon \ .
\]
\begin{itemize}
\item[(i)]{If $p=0$, $s\leq n/2$, $n/s\leq m$ and $\varepsilon \leq 1/(2e)$ then necessarily
\[
\mu \geq \sqrt{\frac{n}{2ms} \log \left( \frac{2 n}{s^2} \log \left( \tfrac{1}{e}-4\varepsilon \right)  +1 \right)} \ .
\]
}
\item[(ii)]{If $p=1$ and $\varepsilon <1/2$ then necessarily
\[
\mu \geq \sqrt{\log \left( \frac{n^2}{s^2 m} \log \left( 4(1-2\varepsilon )^2 +1 \right) +1 \right)} \ .
\]
}
\end{itemize}
\end{theorem}

Considering the case $p=1$, the result above tells us that when $m$ scales like $n/s$, the signal strength needs to scale as $\sqrt{\log (n/s)}$ for detection to be possible.  This is the same scaling that is guaranteed by Theorem~\ref{thm:upper}.  This should come as no surprise, since when $p=1$ we have $\1 \{ A_t \in S^{(t)} \} \sim \Ber (s/n)$ independently for every $t\in [m]$, regardless of the choice of $A_t$.  Hence the resulting measurements $\{ Y_t \}_{t\in [m]}$ follow the same mixture distribution under the alternative, no matter what sampling strategy we use.  Although settings like these have been studied extensively (see \cite{HC_Donoho_2004} and references therein), those works consider asymptotic results.  As such we find it useful to prove a non-asymptotic result for our particular problem, though we point out that this can be simply established by following the steps of the referenced proofs.

Contrasting with this one has the (arguably) more interesting case when the signal is static ($p=0$).  Although the problem of detecting static signals have been the focus of much work (see for instance \cite{Minimax_Ingster_2000, Det_Ingster_2002}), a key difference in our setting is that the sensing actions of the experimenter are not fixed, but are free to choose.  This results in a qualitatively different statement, as the following remark attests.

\begin{remark}\label{rem:NA_lower}
In particular, the first part of the theorem above is interesting in its own right.  It tells us that, for static signals, if the experimenter is free to choose the sensing actions, the signal magnitude needs to scale at least as $\sqrt{\tfrac{n}{sm} \log \tfrac{n}{s^2}}$ for detection to be possible.  It is easy to see that this rate can (almost) be achieved using a sub-sampling scheme: select roughly $n/s$ components at random and collect an equal number of samples of each.  Average the observations for each component separately, and declare a signal if any of these averages is above the threshold $\sqrt{\tfrac{n}{sm} \log \tfrac{n}{s}}$.  Basic calculations show that this procedure has low probability of error.

Contrasting this, the lower bounds of \cite{Minimax_Ingster_2000, Det_Ingster_2002}, which pertain the situation where we measure each component of the vector exactly once, scale as $\sqrt{\log \tfrac{n}{s^2}}$.  Hence, the additional flexibility of where to sample buys us a multiplicative factor of $\sqrt{\tfrac{n}{sm}}$, even though no feedback from the observations is used.  If we can use this feedback, we can also get rid of the log-factor, as shown in \cite{AS_Rui_2012}.
\end{remark}

\begin{remark}
In light of the previous remark, the authors suspect the lower bound in part (i) of the Theorem is slightly loose.  Namely, the term $s^2$ appears to be due to slack in the second moment method in Equation~\ref{eqn:slack}, and it might be possible to replace it by $s$ via a more sophisticated truncation argument.
\end{remark}

The result above tells us that in the regime $m \approx n/s$, the signal strength needs to scale as $\sqrt{\log (n/s^2)}$ for detection to be possible --- approximately the same magnitude as required for $p=1$.  On the other hand Theorem~\ref{thm:upper} guarantees the existence of an adaptive sensing procedure that reliably detects static signals of constant magnitude (in terms of the parameters $n$ and $s$) using roughly $n/s$ measurements.  This shows that adaptive sensing gains over non-adaptive sensing become more pronounced as the speed of change decreases.

Finally we point out once more that the requirements for the signal strength of Theorem~\ref{thm:na_lower} are essentially the same for $p=0$ and $p=1$.  Although we did not succeed in proving a result that holds for any value of $p$ due to technical difficulties, we conjecture that the lower bound or general values of $p$ should interpolate between these two extremes.  In other words, we suspect that the problem difficulty is essentially independent of $p$ in the non-adaptive case when $m$ is of the order (or slightly larger than) $n/s$.  This conjecture is further supported by numerical simulations of testing error probability presented in Section~\ref{sec:sim}.

\begin{proof}[Proof of Theorem~\ref{thm:na_lower}]
\emph{(i):} To prove the claim above for $p=0$ we use the truncated second moment method, an approach suggested by \cite{Ingster_1997} to address problems in the regular second moment method when the distribution of the likelihood ratio under the null has tails that are too heavy (and therefore too large of a second moment).  First, note that
\begin{equation}\label{eqn:TVbound}
\max_{i=0,1} \P_i (\Psi \neq i) \geq \frac{1}{2} \sum_{i=0}^1 \P_i (\Psi \neq i) = \frac{1}{2} \left( 1-\frac{1}{2} \E_0 \left( |L(\vec{Y})-1|\right) \right) \ ,
\end{equation}
where $L(\vec{Y})$ denotes the likelihood-ratio of the observations $\vec{Y}=(Y_1,\ldots,Y_m)$, and $\E_0$ is the expectation taken with respect to the distribution of the observations $\vec{Y}$ under the null.  The second equality is well known (see for instance \cite{Comb_Testing_Lugosi_2010}), and can be easily checked using simple algebraic manipulations.

A common way to proceed is to use either Cauchy-Schwarz's or Jensen's inequality to get
\[
\E_0 \left( |L(\vec{Y})-1|\right) \leq \sqrt{\E_0 \left( (L(\vec{Y})-1)^2 \right)} = \sqrt{\Var_0 (L(\vec{Y}))} \ .
\]
Therefore, to get a lower bound on the risk we need to get a good upper bound on the variance of the likelihood ratio.  This is often referred to as the second moment method.  However, in some cases there is a lot of slack in the bound and the variance is too large to yield interesting results --- so a modification of the above argument is needed.

Let $\mathcal{Y}$ denote the sample space, and let $\tilde{L}(\vec{y}): \mathcal{Y} \to \R$ be an arbitrary function.  Instead of using the Cauchy-Schwarz inequality right away, let us continue the first chain of inequalities as
\begin{align*}
\E_0 \left( |L(\vec{Y})-1|\right) & = \E_0 \left( |L(\vec{Y}) -\tilde{L}(\vec{Y}) + \tilde{L}(\vec{Y}) -1| \right) \\
& \leq \E_0 \left( |\tilde{L}(\vec{Y}) -1| \right) + \E_0 \left( |L(\vec{Y}) -\tilde{L}(\vec{Y})| \right) \\
& \leq \sqrt{\E_0 \left( \tilde{L}(\vec{Y})^2 \right) -2\E_0 \left( \tilde{L}(\vec{Y})\right) +1} + \E_0 \left( |L(\vec{Y}) -\tilde{L}(\vec{Y})| \right) \ .
\end{align*}
Furthermore, if $\tilde{L}(\vec{y})\leq L(\vec{y})$ for every $\vec{y}\in \mathcal{Y}$, then we have
\begin{equation}\label{eqn:truncated_method}
\E_0 \left( |L(\vec{Y})-1|\right) \leq \sqrt{\E_0 \left( \tilde{L}(\vec{Y})^2 \right) -2\E_0 \left( \tilde{L}(\vec{Y})\right) +1} +1 -\E_0 \left( \tilde{L}(\vec{Y}) \right) \ .
\end{equation}
In order to proceed, we need to lower bound $\E_0 (\tilde{L}(\vec{Y}))$ and upper bound $\E_0 (\tilde{L}(\vec{Y})^2)$.  To get a sharp lower bound with this method, we need a good choice for $\tilde{L}(\vec{y})$.  This is often achieved by truncating the original likelihood-ratio by multiplying with the indicator of a well chosen event.

In our setting the likelihood-ratio can be expressed in a convenient way.  Note that under the null the observations are independent standard normal regardless of the sensing actions, hence
\[
\d \P_0 (\vec{y}) = \prod_{t\in [m]} f_0(y_t) \ ,
\]
where $f_\mu (\cdot )$ is the density of a normal random variable with mean $\mu$ and variance 1.  Under the alternative, the density of the observations is a mixture.  Recall that we are considering the case $p=0$ therefore the signal support $S^{(t)}$ does not change over time, namely $S^{(t)}=S$ for all $t\in[m]$.  The conditional density of the observations given the sensing actions $A=(A_1 ,\dots ,A_m )$ and the support $S$ can be written as
\[
\d \P_1 (\vec{y}|A,S) = \prod_{t\in [m]} \left( \1 \{ A_t \in S\} f_\mu (y_t) + \1 \{ A_t \notin S\} f_0 (y_t) \right) \ .
\]
Hence the likelihood-ratio can be expressed as
\[
L(\vec{y}) = \E \left( \exp \left( \sum_{t\in [m]} \1 \{ A_t \in S\} \log \frac{f_\mu (y_t)}{f_0 (y_t)} \right) \right) \ .
\]

Using the second moment method without truncation, one would need to upper bound the second moment of the likelihood ratio above.  Unfortunately, this yields a loose bound on $\mu$.  The reason is that the second moment will be extremely large when the signal is sampled often, even if this event is relatively rare.  In other words, if $\sum_{t\in [m]} \1 \{ A_t \in S\}$ is large one will face problems.  Note that, since the support is chosen uniformly at random,
\[
\E \left( \sum_{t\in [m]} \1 \{ A_t \in S \} \right) = ms/n \ .
\]
However, for certain choices of design $\sum_{t\in [m]} \1 \{ A_t \in S\}$ can be very far from the mean (e.g., if $A_1=\cdots=A_m$ then $\sum_{t\in [m]} \1 \{ A_t \in S\}$ is equal to $m$ with probability $s/n$ and zero otherwise).  This causes the second moment of the likelihood ratio to be extremely large.  To resolve this issue we truncate the likelihood-ratio to exclude these somewhat troublesome instances.

Begin by defining the sets
$$A_{\rm big} = \{ i:\ \sum_{t\in [m]} \1 \{ A_t =i \} > 2ms/n \} \quad\text{ and }\quad A_{\rm small} = [n] \setminus A_{\rm big}\ .$$
In words, for a given sensing design the signal components are divided in two disjoint subsets: one subset contains signal components that are sampled often, whereas the other contains the remaining components.  A simple pigeon hole principle shows that $|A_{\rm big}|\leq n/(2s)$.  Now define
\[
\tilde{L}(\vec{Y}) = \E \left[\left.  \1\{S\subseteq A_{\rm small}\} \exp \left( \sum_{t\in [m]} \1 \{ A_t \in S\} \log \frac{f_\mu (Y_t)}{f_0 (Y_t)}\right)\right| \vec{Y}\right] \ .
\]
Clearly $\tilde{L}(\vec{y}) \leq L(\vec{y})$ for all $\vec{y}\in \mathcal{Y}$, and so we can apply \eqref{eqn:truncated_method} by controlling the first and second moments of $\tilde L(\vec{Y})$.

First note that, since the event $S\subseteq A_{\rm small}$ does not involve the observations $\vec{Y}$ we can easily conclude that
\[
\E_0 \left( \tilde{L}(\vec{Y}) \right) = \P (S\subseteq A_{\rm small}) = \E \left( \P \left(S\subseteq A_{\rm small}|\vec{A} \right) \right) \ ,
\]
where $\vec{A}\equiv (A_1 ,\dots ,A_m )$.  The conditional probability on the right can be lower bounded as
\begin{align*}
\P \left(S\subseteq A_{\rm small}|\vec{A} \right) & = \frac{{|A_{\rm small}| \choose s}}{{n \choose s}} = \frac{|A_{\rm small}| (|A_{\rm small}|-1)\dots (|A_{\rm small}|-s+1)}{n(n-1)\dots (n-s+1)} \\
& \geq \left( \frac{|A_{\rm small}|-s+1}{n-s+1} \right)^s \geq \left( \frac{n\left(1-\tfrac{1}{2s}\right)-s+1}{n-s+1} \right)^s \\
& = \left( 1-\frac{n}{2s(n-s+1)} \right)^s \geq \left( 1-\frac{1}{s} \right)^s\\
& \geq \frac{1}{e}\ ,
\end{align*}
where we used $|A_{\rm small}|\geq n\left(1-\frac{1}{2s}\right)$ and $1\leq s\leq n/2$.

We are left with upper bounding the second moment of $\tilde{L}(\vec{Y})$.  First, note that in the non-adaptive sensing setting $A=(A_1,\dots ,A_m)$ and $S$ are independent.  The proof proceeds by careful conditioning on these random quantities.  We use Jensen's inequality to write
\begin{align*}
\E_0 (\tilde{L}(Y)^2) & = \E_0 \Bigg[ \Bigg(\E \Bigg[\1\{S\subseteq A_{\rm small}\} \exp \Bigg( \sum_{t\in [m]} \1 \{ A_t \in S\} \log \frac{f_\mu (Y_t)}{f_0 (Y_t)}\Bigg)\Bigg| \vec{Y}\Bigg] \Bigg)^2\Bigg]\\
& \leq  \E_0\Bigg[\E\Bigg[\Bigg(\E \Bigg[\underbrace{\1\{S\subseteq A_{\rm small}\} \exp \Bigg( \sum_{t\in [m]} \1 \{ A_t \in S\} \log \frac{f_\mu (Y_t)}{f_0 (Y_t)}\Bigg)}_{h(S,\vec{Y},\vec{A})}\Bigg| \vec{Y},\vec{A}\Bigg]\Bigg)^2\Bigg|\vec{Y}\Bigg]\Bigg]\ .
\end{align*}
At this point it is convenient to introduce an extra random variable $S'$, independent from $S$ and identically distributed.  Then
\begin{align*}
\left(\E \left[\left.  h(S,\vec{Y},\vec{A})\right| \vec{Y},\vec{A}\right]\right)^2 &= \E \left[\left.  h(S,\vec{Y},\vec{A})\right| \vec{Y},\vec{A}\right]\ \E \left[\left.  h(S',\vec{Y},\vec{A})\right| \vec{Y},\vec{A}\right]\\
&= \E \left[\left.  h(S,\vec{Y},\vec{A})h(S',\vec{Y},\vec{A})\right| \vec{Y},\vec{A}\right]\ .
\end{align*}
Therefore we conclude that
\begin{align*}
\E_0\left[\tilde L(\vec{Y})^2\right] &\leq \E_0 \left[\1\{S\subseteq A_{\rm small}\} \1\{S'\subseteq A_{\rm small}\} \exp \left( \sum_{t\in [m]} \left(\1\{ A_t \in S\}+\1\{ A_t \in S'\}\right) \log \frac{f_\mu (Y_t)}{f_0 (Y_t)}\right)\right]\\
&= \E \left[ \E_0\left[\left.  \1\{S,S'\subseteq A_{\rm small}\} \exp \left( \sum_{t\in [m]} \left(\1\{ A_t \in S\}+\1\{ A_t \in S'\}\right) \log \frac{f_\mu (Y_t)}{f_0 (Y_t)}\right)\right|\vec{A},S,S'\right]\right]\\
&= \E \left[\1\{S,S'\subseteq A_{\rm small}\}\prod_{t\in [m]} \E_0\left[\left.  \exp \left(\left(\1\{ A_t \in S\}+\1\{ A_t \in S'\}\right) \log \frac{f_\mu (Y_t)}{f_0 (Y_t)}\right)\right|\vec{A},S,S'\right]\right]\\
&= \E \left[\1\{S,S'\subseteq A_{\rm small}\} \exp\left(\mu^2 \sum_{t\in [m]} \1\{A_t\in S\cap S'\}\right)\right]\ .
\end{align*}
We are now in a good position to finish the bound.  Note that, when $S,S'\subseteq A_{\rm small}$ we have $\sum_{t\in[m]} \1\{A_t=i\}\leq 2ms/n$.  It follows that
\begin{align*}
\E_0\left[\tilde L(\vec{Y})^2\right] &\leq \E\left[\E\left[\left.\1\{S,S'\subseteq A_{\rm small}\} \exp\left(\mu^2 \sum_{i\in[n]} \1\{i\in S\cap S'\}\sum_{t\in[m]} \1\{A_t=i\}\right) \right|\vec{A}\right]\right]\\
&\leq \E\left[\E\left[\left.  \exp\left(\frac{2ms\mu^2}{n} \sum_{i\in[n]} \1\{i\in S\cap S'\}\right) \right|\vec{A}\right]\right]\\
&= \E \left[ \exp \left( \lambda \sum_{i\in[n]} \1\{i\in S\cap S'\} \right) \right]\ ,
\end{align*}
where $\lambda=\frac{2ms\mu^2}{n}$.  The beauty of the last expression is that it no longer involves the sensing actions or the observations, and depends only on the support.  Using the negative association property of $\1 \{ i\in S\cap S'\}$ as introduced in \cite{kumar_1983} we can finally bound the second moment of the truncated likelihood as
\begin{align}
\E_0 \left(\tilde{L}(\vec{Y})^2\right) & \leq \E\left[\exp\left(\lambda\sum_{i\in[n]} \1\{i\in S\cap S'\}]\right)\right]\nonumber\\
& = \E\left[\prod_{i\in[n]} e^{\lambda \1\{i\in S\cap S'\}}\right] \leq \prod_{i\in[n]} \E\left[ e^{\lambda \1\{i\in S\cap S'\}}\right]\nonumber\\
& = \left( 1+ \frac{s^2}{n^2} \left( e^\lambda -1 \right) \right)^n = \left( 1+ \frac{s^2}{n^2} \left( e^{2\mu^2 ms/n} -1 \right) \right)^n \label{eqn:slack}\ .
\end{align}
We have now all the ingredients needed to complete the proof.  Note that, on one hand, if $\max_{i=0,1} \P_i (\Psi \neq i) \leq \varepsilon$ then necessarily $\E_0[|L(\vec{Y})-1|]\geq 2-4\varepsilon$.  On the other hand, from \eqref{eqn:truncated_method} we know that
\begin{align*}
\E_0[|L(\vec{Y})-1|] &\leq \sqrt{\E_0 \left( \tilde{L}(\vec{Y})^2 \right) -2\E_0 \left( \tilde{L}(\vec{Y})\right) +1} +1 -\E_0 \left( \tilde{L}(\vec{Y}) \right)\\
&< \left(1+ \frac{s^2}{n^2} \left( e^{2\mu^2 ms/n} -1\right) \right)^{n/2} +\frac{1}{2}\ .
\end{align*}
This means that
$$\frac{s^2}{n^2}\left(e^{2\mu^2 ms/n} -1\right) > \left(\frac{3}{2}-4\varepsilon\right)^{2/n}-1\geq \frac{2}{n}\log\left(\frac{3}{2}-4\varepsilon\right)\ ,$$
where the last inequality uses the fact that $x-1 \geq \log x$.  The final result ensues by simple algebraic manipulation.

\vspace{0.2cm}
\emph{(ii):} Proving the claim for $p=1$ requires considerably less technical effort.  In particular we can use the original second moment method, without truncation.  Therefore, we simply need to upper bound the second moment of the likelihood-ratio.

Using essentially the same calculations as before, we get
\[
\E_0 \left( L(Y)^2 \right) = \E \left[ \exp \left( \mu^2 \sum_{t\in [m]} \1 \{ A_t \in S^{(t)} \cap S'^{(t)} \} \right) \right] \ .
\]
When $p=1$ we have that $\1 \{ A_t \in S^{(t)} \cap S'^{(t)} \} \sim \Ber (s^2 /n^2 )$ and these random variables are independent, so we can simply evaluate the above expression and get 
\[
\E_0 \left( L(Y)^2 \right) = \left( 1+ \frac{s^2}{n^2} \left( e^{\mu^2} -1 \right) \right)^m \ .
\]

Plugging this into the inequalities above (not using the truncation), we get
\[
\mu \geq \sqrt{\log \left( \frac{n^2}{s^2} \left( \sqrt[m]{4(1-2\varepsilon )^2}-1 \right) +1 \right)} \ .
\]
The desired result follows by using $x-1 \geq \log x$.
\end{proof}


\subsection{Adaptive sensing} \label{sec:lower_a}

In the adaptive sensing setting, the decision where to sample at time $t$ can depend on information gleaned up to that point.  For the static case ($p=0$) the fundamental limits of the detection problem using adaptive sensing have been studied in \cite{AS_Rui_2012}.  Those lower bounds are derived for a slightly more general setting than the one considered here, in that the total precision of the measurements is constrained, but not the total number of measurements.  Nevertheless, this bound is still valid in our setting, and states that for any adaptive sensing and testing procedure $\Psi$ if
\[
\max \{ \P_0 (\Psi \neq 0 ), \P_1 (\Psi \neq 1) \} \leq \varepsilon
\]
then necessarily
\[
\mu \geq \sqrt{\frac{2 n}{sm} \log \frac{1}{2\varepsilon}} \ .
\]

In the regime $m \approx n/s$ the bound states that the signal strength needs scale as $\sqrt{\log (1/\varepsilon )}$.  This coincides (up to constants) with the bound of Theorem~\ref{thm:upper} when $p\leq 2/\log (n/s)$.  This tells us that when the signal changes slowly enough, the problem is essentially non-dynamic in nature.

On the other extreme end of the spectrum is the case $p=1$.  We have seen previously that in this case the non-adaptive and adaptive sensing settings are identical, by virtue of the fact that $\1 \{ A_t\in S^{(t)} \} \sim \Ber (s/n)$ for every $t\in [m]$ and independent, regardless of the choice of $A_t$.

What remains to be understood are the fundamental limits for the intermediate regime.


\subsubsection{Non-extreme dynamics ($p\in (0,1)$)}

For general values of $p$ we start by considering the case $s=1$, which we call the 1-sparse case.  This case is considerably simpler to analyze than the general $s$-sparse setting, as now whenever the active component changes the entire signal resets.  This effectively creates a number of independent static signals on the time horizon.

\begin{theorem}\label{thm:1sparse_lower}
Consider the setup in Section~\ref{sec:setup} and suppose there exists a test $\Psi$ such that
\[
\max_{i=0,1} \P_i (\Psi \neq i) \leq \varepsilon \ .
\]
\begin{itemize}
\item[(i)]{The signal strength must satisfy
\[
\mu \geq \sqrt{\frac{2n}{sm} \log \frac{1}{4\varepsilon}} \ .
\]
}
\item[(ii)]{
When $s=1$ and $p\geq 8/m$, then necessarily
\[
\mu \geq \sqrt{\frac{p}{2c} \log \left( \log \left( \left( \tfrac{5}{4} -4\varepsilon \right)^2 +\tfrac{1}{2} \right) \frac{p^2 n^2}{4 c^2 m} +1 \right)} \ ,
\]
with $c=6+3\log 2$.}
\end{itemize}
\end{theorem}


We provide the proof of Theorem~\ref{thm:1sparse_lower} at the end of the section.  Part (i) holds regardless of the values of $p$ and $s$, so it is necessarily loose when $p$ is large.  On the other hand part (ii) already captures the role of the rate of change $p$, and it is the main contribution in this result.

Let us compare the above bound on $\mu$ with the guarantees for Algorithm~\ref{basic_algorithm} proved in Theorem~\ref{thm:upper}.  Note that $c$ and $\varepsilon$ are constants.  Thus the bound on the signal strength in the above result scales as $\sqrt{p \log(p^2 n^2/m)}$.  Recall that we are interested in the regime $m\approx n/s$ and that $s=1$, as we are considering the 1-sparse case.  In that setting the bound above scales as $\sqrt{p \log (p^2 n)}$.  Also note that the scaling of the performance guarantee of Theorem~\ref{thm:upper} matches that of the lower bound from \cite{AS_Rui_2012} when $p<1/\log n$.  Hence we only need to assess the result of Theorem~\ref{thm:1sparse_lower} when $p\geq 1/\log n$.  In this case, the scaling of that bound is at least as big as $\sqrt{p(\log n - 2\log \log n)} \approx \sqrt{p \log n}$.  This shows near-optimality of the algorithm proposed in Section~\ref{sec:upper}, in terms of its scaling in the parameters $n$ and $p$.

Due to technical reasons we were unable to generalize the result for signals of sparsity greater than one.  As noted above, a key feature of the 1-sparse case is that the signal decouples into independent static signals over time.  This key property is lost when we consider signals with sparsity greater than one, and this proves to be a major obstacle to obtain a rigorous formal proof.  However, we conjecture that a similar result to the one above holds for $s$-sparse signals, with $n$ replaced by $n/s$.  The heuristic behind this is that a general $s$-sparse signal of dimension $n$ should behave very much like an $s$-fold concatenation of an 1-sparse signal of dimension $n/s$, when viewed through the lens of one measurement per time-index (one expects this to actually be a statistical reduction, and this problem should be statistically easier than the original one).  For such a signal the result above would follow directly with the signal dimension $n$ replaced by $n/s$.

\begin{conjecture}
When $p\geq 8/m$, if the risk of an adaptive sensing and test procedure is less or equal to $\varepsilon$ then necessarily
\[
\mu \geq \sqrt{\frac{p}{2c} \log \left( \log \left( \left( \tfrac{5}{4} -4\varepsilon \right)^2 +\tfrac{1}{2} \right) \frac{p^2 n^2}{4 c^2 s^2 m} +1 \right)} \ ,
\]
with $c=6+3\log 2$.
\end{conjecture}

\begin{proof}[Proof of Theorem~\ref{thm:1sparse_lower}]
We prove the two parts of the statement separately.

\vspace{0.2cm}
\emph{(i):} The proof is very similar to that of Theorem~3.1 in \cite{AS_Rui_2012}, with small modifications to be able to deal with dynamically evolving signals (which actually simplify the argument).  By Theorem~2.2 of \cite{Tsybakov_2009} we have
\begin{equation}\label{eqn:tsybakov}
\inf_{\Psi} \max_{i=0,1} \P_i (\Psi \neq i) \geq \frac{1}{4} e^{- \KL (\P_0 \| \P_1 )} \ ,
\end{equation}
where $\KL (\P_0 \| \P_1 )$ denotes the Kullback-Leibler divergence between the distribution of the data $\bY$ under the null and alternative respectively.  This divergence can be simply upper bounded using Jensen's inequality as
\begin{align*}
\KL (\P_0 \| \P_1 ) & = \E_0 \left[ -\log L(\bY ) \right] \\
& \leq \E_0 \left[ \E \left[ \left.  -\sum_{t\in [m]} \1 \{ A_t \in S^{(t)} \} \log \frac{f_\mu (Y_t)}{f_0 (Y_t)} \right| \bY \right] \right] \ .
\end{align*}

Changing the order of integration and expanding the densities $f_\mu (\cdot )$ and $f_0 (\cdot )$ we get
\[
\KL (\P_0 \| \P_1 ) \leq \frac{\mu^2}{2} \E \left[ \sum_{t\in [m]} \1 \{ A_t \in S^{(t)} \} \right] = \frac{\mu^2}{2}\frac{sm}{n} \ ,
\]
where the last step follows from the symmetry of the supports.  In particular note that $\E [ \1 \{ A_t \in S^{(t)}\} | A_t ] = s/n$ for every $t\in [m]$.  Plugging this bound into the right side of \eqref{eqn:tsybakov}, using that the left side of \eqref{eqn:tsybakov} is at most $\varepsilon$ due to our assumption, and rearranging concludes the proof of the first claim.

\vspace{0.2cm}
\emph{(ii):} We use the truncated second moment method, as in the proof of Theorem~\ref{thm:na_lower}.  Recall that, from \eqref{eqn:TVbound} and \eqref{eqn:truncated_method} , we have
\[
2 \max_{i=0,1} \P_i (\Psi \neq i) \geq 1- \frac{1}{2} \left( \sqrt{\E_0 \left( \tilde{L}(Y)^2 \right) -2\E_0 \left( \tilde{L}(Y)\right) +1} +1 -\E_0 \left( \tilde{L}(Y) \right) \right) \ ,
\]
for any function $\tilde{L}(\cdot )$ satisfying $\tilde{L}(\by ) \leq L(\by ),\ \forall \by \in \mathcal{Y}$, where $L(\cdot )$ is the likelihood function.

\begin{figure}
\centering
\includegraphics[scale=0.38]{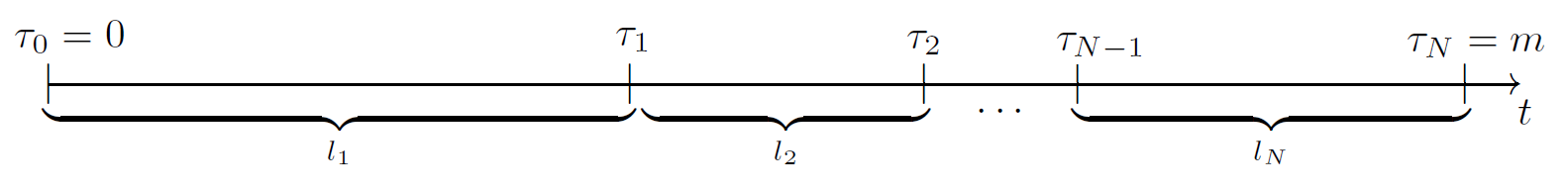}

\caption{Illustration of the notation introduced for the proof part (ii) of Theorem~\ref{thm:1sparse_lower}.}
\label{fig:notation}
\end{figure}

To aid the presentation we begin by introducing some convenient notation, illustrated in Figure~\ref{fig:notation}.  Recall that the variables $\theta^{(t)}_{i} \displaystyle{\mathop{\sim}^{\text{i.i.d.}}}\ \Ber (p)$, $t\in [m],\ i\in [s]$ identify the change points of the signal.  Since now we are dealing with the 1-sparse case we have one variable per time index, so in what follows we drop the subscript from the previous notation.  Furthermore, note that our time horizon is $m$, so we enforce $\theta^{(m)} =1$ as this does not change the model and it is convenient for the presentation.

Let the total number of change points over the time horizon be $N = \sum_{t\in [m]} \1 \{ \theta^{(t)} =1\}$.  Note that $N-1 \sim \Bin (m-1,p)$.  Let $\tau_0 = 0$ and for $j\in N$ let $\tau_j = \min \{ t>\tau_{j-1}:\ \theta^{(t)} =1 \}$ denote the time instances when the signal changes (so $\tau_N =m$), as illustrated in Figure~\ref{fig:notation}.  Note that on the time intervals $[ \tau_j +1,\tau_{j+1} ]$ the signal is static.  Let $l_j=\tau_{j+1}-\tau_{j}$ denote the length of these intervals, and $S_j,\ j\in [N]$ denote the correspoding signal support.  Finally, for any $t\in [m]$ let the number of change points up to time $t$ be ${N(t) = \max \{ j: \tau_j \leq t \}}$.  It is important to note that the random variables $\theta^{(t)}$ completely determine the variables $\tau_j$, $N(t)$ and $N$.

Let us first explicitly write the likelihood of the observations in the model under consideration.  We use the shorthand notation $\by = \{ y_t \}_{t\in [m]}, \bA =\{ A_t \}_{t\in [m]}, \bS =\{ S^{(t)} \}_{t\in[m]}, \btheta =\{ \theta_t \}_{t\in [m]}$.  As before, the density of $\by$ under the alternative is a mixture.  In particular, denoting the density of $N(\mu ,1)$ by $f_\mu$, the conditional density of $\by$ can be written as
\begin{align*}
\d \P_1 (\by |\bA ,\bS ) & = \prod_{t\in [m]} \left( \1 \{ A_t \in S^{(t)} \} f_\mu (y_t) + \1 \{ A_t \notin S^{(t)} \} f_0 (y_t) \right) \\
& = \prod_{j\in [N]} \prod_{t= \tau_{j-1}+1}^{\tau_j} \left( \1 \{ A_t \in S_j \} f_\mu (y_t) + \1 \{ A_t \notin S_j \} f_0 (y_t) \right) \ .
\end{align*}
Hence, the likelihood ratio is
\[
L(\by ) = \E \left[ \E \left[ \left.  \exp \left( \sum_{j\in [N]} \sum_{t=\tau_{j-1}+1}^{\tau_j} \1 \{ A_t\in S_j \} \log \frac{f_\mu (y_t)}{f_0 (y_t)} \right) \right| \btheta ,\bA \right] \right] \ ,
\]
where conditioning on $\btheta$ and $\bA$ is done in order to conveniently define $\tilde L(\by)$.  Consider the event
\[
\Omega_c = \left\{ \forall j:\ l_j \leq 2c/p \right\} \ ,
\]
with some fixed $c>0$.  This event says that the signal is never static for a time longer than $2c/p$.  Note that this event is determined exclusively by the variables $\{ \theta_t \}_{t\in [m]}$.  We define the truncated likelihood as
\[
\tilde{L}(\by ) = \E \left[ \1\{\Omega_c\} \E \left[ \left.  \exp \left( \sum_{j\in [N]} \sum_{t=\tau_{j-1}+1}^{\tau_j} \1 \{ A_t\in S_j \} \log \frac{f_\mu (y_t)}{f_0 (y_t)} \right) \right| \btheta ,\bA \right] \right] \ .
\]

As in the proof of Theorem~\ref{thm:na_lower}, we need to upper bound $\E_0 \left( \tilde{L}(\bY)^2 \right)$ and lower bound $\E_0 \left( \tilde{L}(\bY)\right)$.  We start with the latter.  Since the event $\Omega_c$ only involves the variables $\btheta$, we have
\[
\E_0 \left( \tilde{L}(\bY)\right) = \P (\Omega_c ) \ .
\]
We have the following result, the proof of which is presented in the Appendix.

\begin{lemma}\label{lem:event}
Consider the event
\[
\Omega_c = \left\{ \forall j:\ l_j \leq 2c/p \right\} \ .
\]
In the model described above $\P (\Omega_c ) >1/4$ whenever $c\geq 6 + 3 \log 2$ and $p \geq 8/m$.
\end{lemma}

According to Lemma~\ref{lem:event}, we have an appropriate bound for $\E_0 \left( \tilde{L}(Y)\right)$ when $c\geq 6 + 3 \log 2$.  All that remains is to derive an upper bound on the truncated second moment.  This can be done much the same way as in the proof of Theorem~\ref{thm:na_lower}.  Using Jensen's inequality, we have
\[
\E_0 \left[ \tilde{L}(\bY )^2 \right] \leq \E_0 \left[ \E \left[ \1\{\Omega_c\} \E \left[ \left.  \exp \left( \sum_{j\in [N]} \sum_{t=\tau_{j-1}+1}^{\tau_j} \1 \{ A_t\in S_j \} \log \frac{f_\mu (y_t)}{f_0 (y_t)} \right) \right| \btheta ,\bA \right]^2 \right] \right] \ .
\]
Note that given $\btheta$, the $S_j \sim \Unif ([n])$ and independent for $j\in [N]$.  Let $\{ S'_j \}_{j\in [N]}$ be an independent copy of $\{ S_j \}_{j\in [N]}$.  Following the same reasoning as in Theorem~\ref{thm:na_lower} we can write the square of the conditional expectation above as the product of two expectations using the random variables $\{S_j, S'_j \}_{j\in [N]}$, and change the order of the expectations to get
\[
\E_0 \left[ \tilde{L}(\bY )^2 \right] \leq \E \left[ \1\{\Omega_c\} \E \left[ \left.  \exp \left( \mu^2 \sum_{j\in [N]} \sum_{t=\tau_{j-1}+1}^{\tau_j} \1 \{A_t \in S_j \cap S'_j \} \right) \right| \btheta ,\bA \right] \right] \ .
\]

So far we have not taken into account the fact that we are allowed an adaptive design.  This is captured by the crude bound below.
\[
\sum_{t=\tau_{j-1}+1}^{\tau_j} \1 \{A_t \in S_j \cap S'_j \} \leq l_j \1 \{ \exists t\in [\tau_{j-1}+1,\tau_j ]:\ A_t \in S_j \cap S'_j \} \ .
\]
Informally this means that, if the used design ``hits'' the signal at any place in the interval $[\tau_{j-1}+1,\tau_j ]$ it is assumed the design hit the signal in the entire interval (capturing more information).  Furthermore
\[
\P \left(\left.  \exists t\in [\tau_{j-1}+1,\tau_j ]:\ A_t \in S_j \cap S'_j\right|\bA,\btheta \right) = \P \left(\left.  S_j \in \{ A_t :\ t\in [\tau_{j-1}+1,\tau_j ] \} \right|\bA,\btheta \right)^2 \ .
\]
However, $|\{ A_t:\ \tau_{j-1} +1 \leq t \leq \tau_j \}| \leq \tau_j-\tau_{j-1} := l_j$ thus the probability above is bounded from above by $l_j^2 /n^2$.

Using all this yields
\begin{align*}
\E_0 \left[ \tilde{L}(\bY )^2 \right] & \leq \E \left[ \1\{\Omega_c\} \prod_{j\in [N]} \E \left[ \left.  \exp \left( l_j \mu^2 \1 \{\exists t\in [\tau_{j-1}+1,\tau_j ]:\ A_t \in S_j \cap S'_j \} \mu^2 \right) \right| \btheta ,\bA \right] \right] \\
& \leq \E \left[ \1\{\Omega_c\} \prod_{j\in [N]} \left( 1+ \frac{l_j^2}{n^2} \left( e^{l_j \mu^2} -1 \right) \right) \right] \ .
\end{align*}

The last expression is readily upper bounded by the fact that $N\leq m$.  Although this is a {crude bound}\footnote{In principle one can recall that $N-1\sim\Bin(m-1,p)$ and proceed from there, although it will overcomplicate the derivation.  In any case, this will at most allow us to replace the term $p^2$ by $p$ inside the logarithm in the statement of the theorem, which is not very relevant.} it is enough for our purposes.  Also, on the event $\Omega_c$ we have the upper bound $l_j \leq 2c/p$ for every $j\in [N]$.  We conclude that
\[
\E_0 \left( \tilde{L}(\bY )^2 \right) \leq \left( 1+ \frac{4c^2}{p^2 n^2} \left( e^{2c\mu^2 /p} -1 \right) \right)^m \ .
\]



Combining our results yields that if there exists a test for which $\max_{i=0,1} \P (\Psi \neq i) \leq \varepsilon$, we must have
\[
\sqrt{\left( 1+ \frac{4c^2}{p^2 n^2} \left( e^{2c\mu^2 /p}-1 \right) \right)^m -\frac{1}{2}} +\frac{3}{4} \geq 2-4\varepsilon \ .
\]
Rearranging gives
\[
\frac{4c^2}{p^2 n^2} \left( e^{2c\mu^2 /p}-1 \right) \geq \sqrt[m]{\left( \frac{5}{4} - 4\varepsilon \right)^2 + \frac{1}{2}}-1 \ .
\]
Using the inequality $\log x \leq x-1$ on the right hand side, and rearranging concludes the proof.
\end{proof}


\section{Numerical evaluation of the non-adaptive lower bound}\label{sec:sim}

Although the lower bound in Theorem~\ref{thm:na_lower} only deals with the extreme cases $p\in \{ 0,1 \}$, we conjecture that in the regime $m\approx n/s$ the same scaling of $\mu$ is necessary for reliable detection, regardless of the value of $p$.  To corroborate this conjecture we provide a brief section of numerical experiments.  We numerically estimate the right hand side of \eqref{eqn:TVbound}, which is a lower bound on the maximal probability of error.  We do so for several values of $p\in [0,1]$, and for each $p$ we plot the value of the lower bound as a function of $\mu$.

Note that the sampling strategy has a large impact on the value in question.  We know that when $p=0$ a sub-sampling scheme is near-optimal (see Remark~\ref{rem:NA_lower}), and so it should also be reasonable for small values of $p$.  On the other hand, the sampling strategy is irrelevant for $p=1$, and probably essentially irrelevant for large $p$.  This motivates using a sub-sampling scheme in all the experiments.

Furthermore, note that unless we sample $c\cdot n/s$ different components, the probability $\P_1 (\forall t\in [m]:\ A_t \notin S^{(t)})$ can not be small.  To ensure an upper bound of $\varepsilon$ on the previous probability, we need to choose $c \equiv c(\varepsilon ) = \log (1/\varepsilon )$.

Considering all the above, we set up our experiment as follows.  We set $n=5000, s=\lceil n^{1/4} \rceil =9$ and $m=c(\varepsilon )n/s$ with $\varepsilon =0.05$.  In this case, sub-sampling reduces to measuring $m$ randomly selected components (one measurement each).  We note that we experimented using multiple values of $s$ across a wide range of sparsity levels, but found qualitatively the same result in all cases.

Based on previous work concerning the sparse-mixture model (e.g.  \cite{HC_Donoho_2004}) we expect the lower bound to reach the value $\varepsilon$ when $\mu \approx \sqrt{2\log (n/s)}$.  Hence, we set $\mu_t \approx t\cdot \sqrt{2\log (n/s)}$, and plot the r.h.s.  of \eqref{eqn:TVbound} as a function of $t$.

The left panel of Figure~\ref{fig:NA_lower} seems to support our conjecture that the problem difficulty is independent of $p$ in the regime $m\approx n/s$, as all the curves are on top of each other.  Furthermore, since there is always a non-negligible chance of not sampling a signal component, the lower bound is bounded away from zero, even as $\mu_t$ grows large.

\begin{figure}[t]
\centering
\subfigure[]{\includegraphics[scale=0.52]{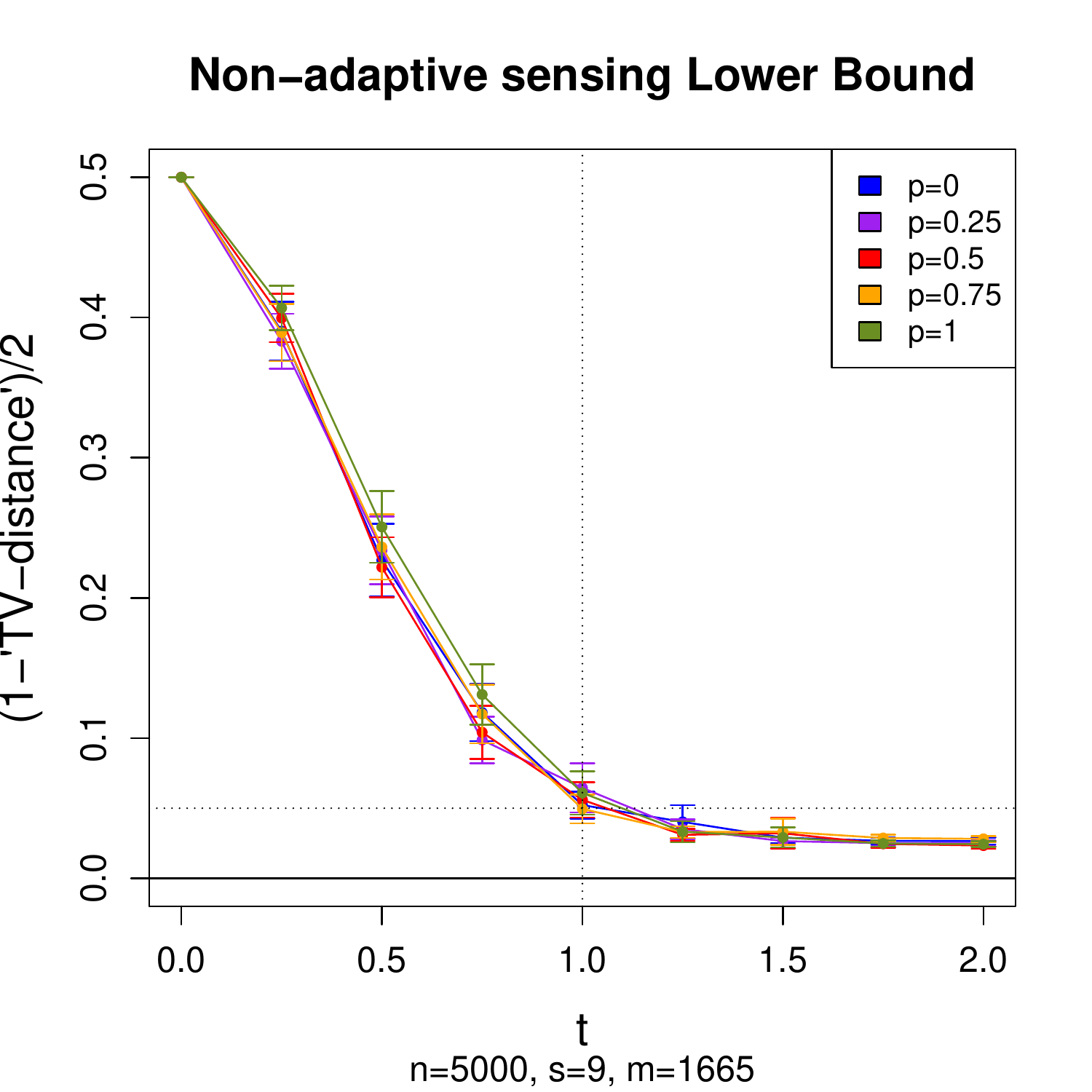}}
\subfigure[]{\includegraphics[scale=0.52]{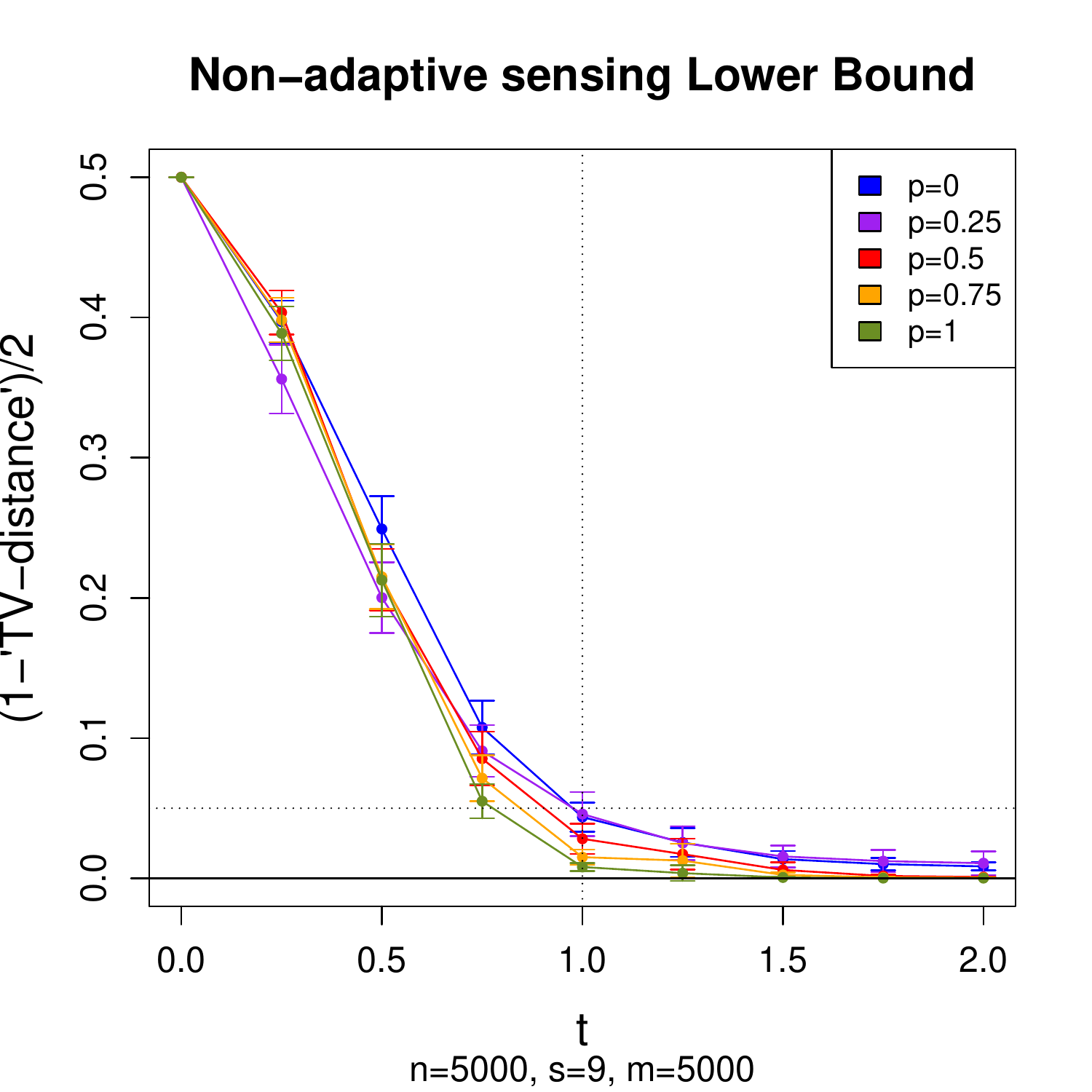}}
\caption{R.h.s.  of \eqref{eqn:TVbound} as a function of the signal strength $\mu = t\cdot \sqrt{2 (2c(\varepsilon )n /sm) \log (n/s)}$, left panel: $m=c(\varepsilon ) \cdot n/s$; right panel: $m=n$.  Curves represent different values of $p$: 0 - blue; 0.25 - purple; 0.5 - red; 0.75 - yellow; 1 - green.  Plotted values are the averages based on 100 simulations, and error bars have total length $4$ times the standard error (approximate two-sided $95 \%$ confidence intervals).  Horizontal dashed line at $0.05$.}
\label{fig:NA_lower}
\end{figure}

To contrast this, we present another simulation with the same setup, except that the number of measurements $m \gg n/s$.  In particular, we set $m=n$, but otherwise use the same parameters.  Note that in this case, sub-sampling amounts to sampling $c(\varepsilon ) n/s$ randomly chosen components, but now we sample each of these $m/(c(\varepsilon ) n/s)$ consecutive times.

To keep the two plots on the same horizontal scale, we set $\mu_t = t\cdot \sqrt{ (2c(\varepsilon )n /sm) \log (n/s)}$ in the right panel of Figure~\ref{fig:NA_lower}.  It seems that in this case, the curves are no longer on top of each other, suggesting that the value of $p$ has an impact on the problem difficulty.  Surprisingly, the curve corresponding to $p=1$ is the one that descends the fastest, though the difference is only marginal.  Though the cause of this is unclear, a possible reason might be that for faster signals the chance of not sampling active components at all is diminished, an effect that is more pronounced when $m$ is large.

In any case, this shows that in the regime $m\gg n/s$ the speed of change might have a non-trivial effect on the problem difficulty.  Exploring this is out of the scope of this work, but might be an interesting topic of future research.


\section{Final remarks} \label{sec:end}

In this paper we studied the problem of the detection of signals that evolve dynamically over time.  We introduced a simple model for the evolution of the signal that allowed us to explicitly characterize the difficulty of the problem with a special regard to the effect of the speed of change.  We also showed the potential advantages that adaptively collecting the observations bring to the table and showed that these are more and more pronounced as the speed of change decreases, which is in line with previous results dealing with signal detection using adaptive sensing.  The lower bounds derived in this paper provide a clear picture of the role of the rate of change parameter $p$, but unfortunately still do not span the entire range of problems we would like to consider (e.g.  Theorem~\ref{thm:na_lower} applies only to $p=0,1$ and part~(ii) of Theorem~\ref{thm:1sparse_lower} applies only to $s=1$).  The latter difficulties appear to be mostly technical and the authors suspect these might be possible to address with carefully chosen reductions.  Our contributions merely scratch the surface of this interesting problem, and below we highlight a few interesting directions for future work in this regard.\\

\textbf{Large vs.  small sample regimes:} in this work we focus primarily on the case $m\approx n/s$, which may be deemed as the small sample regime.  When the number of measurements $m$ is significantly larger the type of tests and performance tradeoffs will likely be different, even under the non-adaptive sensing paradigm.  For instance, we expect the signal dynamics to have an effect on performance, meaning that it is easier to detect signals non-adaptively when $p$ is smaller.  Other interesting questions arise in that setting as well --- what is the optimal non-adaptive sensing design?  These questions become even more intriguing when one considers adaptive sensing.\\

\textbf{Restricted dynamics:} in the model considered in this paper when signal components change they can move to any unoccupied location in the signal vector.  This assumption simplifies the setup, but in some applications might be too unrestrictive.  For instance, if signal components can only move to adjacent locations at each time step the effect of the speed of change will likely be less pronounced in the difficulty of detection (at least for adaptive sensing).  Understanding the effect of such restrictions could prove valuable in certain applications, such as detection of a disease outbreak in a network, besides being interesting from a theoretical point of view.\\

\textbf{Structures:} in certain situations the signal support can be assumed to have structure to it, for instance all anomalous items might be consecutive or have some other pattern.  In some cases the structure of the support has a huge effect on the difficulty of the problems of detection and recovery (see for instance \cite{AS_structured_Rui_ET_2013,ACS_structured_Rui_ET_2014}).  How structural restrictions affect these tasks for dynamically evolving signals could be a fruitful avenue of research.\\

\textbf{Support recovery:} another common question in such settings is how well can we estimate the support of a signal.  That is, instead of deciding only if there are anomalous items or not, we need to determine which of the items are anomalous.  This is also an interesting problem to study for dynamically evolving signals, although a precise formulation of the objective and performance metric for such estimators is less immediate than for static signals.

\section*{Acknowledgements}

This work was partially supported by a grant from the {\em Nederlandse organisatie voor Wetenschappelijk Onderzoek} (NWO 613.001.114). We are very grateful for the comments of the two anonymous referees, which helped improving the presentation.

\section*{Appendix}

\begin{proof}[Proof of Lemma~\ref{lem:event}]
We write
\begin{align*}
\P (\Omega_c ) & \geq \E \left( \{ N-1 >mp/2\} \cap \{ \forall j:\ l_j \leq cm/N \} \right) \\
& = \E \left[ \1 \{ N-1 > mp/2 \} \E \left[ \1 \{ \forall j:\ l_j\leq cm/N \} \ \big| N \right] \right] \ ,
\end{align*}

We first lower bound the inner conditional probability.  Note that if $N\leq c$ this probability is one (since $cm/N\geq m$ and $l_j \leq m$ by definition).  When $N>c$, we will upper bound the probability of the complementary event.

Note that given $N$ the distribution of $\btheta$ is uniform from the set of $0-1$ sequences of length $m$ containing exactly $N$ ones, and for which also $\theta_m =1$.  Hence, to upper bound $\P (\exists j:\ l_j > cm/N)$, we simply need to count the number of sequences described above for which we have a long block.

We can get an upper bound on this count in the following way.  First note that since the last element of the sequence is always one, we can simply think of sequences of length $[m-1]$ containing $N-1$ ones.  Consider an interval of length $cm/N$ in the set $[m-1]$.  Now consider the sequences containing $N-1$ ones, and for which there are no ones in the aforementioned interval.  Note that for all such sequences the existence of at least one long interval is guaranteed.  We can simply count how many $0-1$ sequences can be generated like this.  This number is an upper bound on the number of $0-1$ sequences that have $N$ ones, the last element of the sequence is one and for which $\exists j:\ l_j > cm/N$.

We thus have
\begin{align*}
\P (\exists j:\ l_j> cm/N |N) & \leq (m-cm/N) \frac{{m-cm/N \choose N-1}}{{m-1 \choose N-1}} \\
& = (m-cm/N) \frac{(m-cm/N)(m-cm/N -1)\dots (m-cm/N -N+2)}{(m-1)(m-2)\dots (m-N+1)} \\
& \leq \frac{m}{m-1}(1-c/N) \left( \frac{m-cm/N}{m-2} \right)^{N-2} \\
& < \left( \frac{m-cm/N}{m-2} \right)^{N-2} \ .
\end{align*}
Now consider the logarithm of the expression above.  Using $\log (1+x) \leq x$, we get
\begin{align*}
\log \P (\exists j:\ l_j> cm/N |N) & < (N-2) \left( \log \frac{m}{m-2} + \log (1-c/N) \right) \\
& \leq (N-2) \left( \frac{2}{m-2} - \frac{c}{N} \right) \\
& \leq -\log{2} \ ,
\end{align*}
whenever $c\geq 6+3\log 2$, using the fact that $3\leq c\leq N\leq m$.

Hence $\P (\Omega_c ) \geq \P (N-1 >mp/2)/2$.  All that remains is to use the fact that $N-1 \sim \Bin (m-1,p)$.  For instance Chebyshev's inequality yields
\[
\P (N-1 \leq mp/2) \leq \frac{4(m-1)p(1-p)}{(mp)^2} \leq 1/2 \ ,
\]
when $p\geq 8/m$ and so the claim is proved.
\end{proof}

\bibliographystyle{chicago}
\bibliography{Dynamics_references}

\end{document}